\documentclass[1pt]{article}
\usepackage[latin9]{inputenc}
\usepackage{geometry}
\usepackage{xcolor}
\usepackage{mathrsfs}
\usepackage{amsmath}
\usepackage{amssymb}
\usepackage[unicode=true,pdfusetitle,
 bookmarks=true,bookmarksnumbered=false,bookmarksopen=false,
 breaklinks=false,pdfborder={0 0 1},backref=false,colorlinks=false]
 {hyperref}

\makeatletter

\DeclareTextSymbolDefault{\textquotedbl}{T1}

\numberwithin{equation}{section}

\@ifundefined{date}{}{\date{}}
\usepackage{amsthm}\usepackage{amsfonts}\usepackage{cases}\usepackage{mathrsfs}
\usepackage{url}
\usepackage{authblk}

\allowdisplaybreaks

\theoremstyle{plain}
\newtheorem{thm}{Theorem}[section]
\newtheorem{pro}[thm]{Problem}\newtheorem{lem}[thm]{Lemma}\newtheorem{prop}[thm]{Proposition}\theoremstyle{definition}
\newtheorem{defn}{Definition}[section]
\newtheorem{ass}{Assumption}[section]
\newtheorem{rmk}{Remark}[section]

\@addtoreset{equation}{section}

\DeclareMathOperator*{\essinf}{ess\,inf}

\makeatother

\begin{document}

\title{Backward Stochastic Riccati Equation with Jumps associated with Stochastic
Linear Quadratic Optimal Control with Jumps and Random Coefficients
\thanks{This work was supported by the National Natural Science Foundation
of China (No.11871121, 11701369) and the Natural Science Foundation
of Zhejiang Province for Distinguished Young Scholar (No.LR15A010001).
The second auther gratefully acknowledges finincial support from Région
Pays de la Loire throught the grant PANORisk.}}

\author{Fu Zhang\thanks{College of Science, Shanghai of University for Science and Technology,
Shanghai, 200093, China. Email: fugarzhang@163.com. }, Yuchao Dong\thanks{Universit\'e d'Angers, D\'epartement de Math\'ematiques, 2, Bd
Lavoisier, 49045 Angers Cedex 01, France and Department of Mathematics,
Fudan University, Shanghai 200433, China, Email: ycdong@fudan.edu.cn}, Qingxin Meng\thanks{Corresponding author. {\small{}Department of Mathematics, Huzhou University,
Zhejiang 313000, China. }E-mail: mqx@zjhu.edu.cn. } }
\maketitle
\begin{abstract}
In this paper, we investigate the solvability of matrix valued Backward
stochastic Riccati equations with jumps (BSREJ), which is associated
with a stochastic linear quadratic (SLQ) optimal control problem with
random coefficients and driven by both Brownian motion and Poisson
jumps. By dynamic programming principle, Doob-Meyer decomposition
and inverse flow technique, the existence and uniqueness of the solution
for the BSREJ is established. The difficulties addressed to this issue
not only are brought from the high nonlinearity of the generator of
the BSREJ like the case driven only by Brownian motion, but also from
that i) the inverse flow of the controlled linear stochastic differential
equation driven by Poisson jumps may not exist without additional
technical condition, and ii) how to show the inverse matrix term involving
jump process in the generator is well-defined. Utilizing the structure
of the optimal problem, we overcome these difficulties and establish
the existence of the solution. In additional, a verification theorem
for BSREJ is given which implies the uniqueness of the solution.
\end{abstract}

\section{Introduction}

\subsection{Framework and Preliminary}

We start with a stochastic basis $(\Omega,\mathcal{F},\mathscr{F},\mathbb{P})$
with a finite time horizon $T<\infty$ and a filtration $\mathscr{F}:=\{{\mathscr{F}}_{t}|t\in[0,T]\}$
satisfying the usual conditions of right continuity and completeness,
such that we can and do take all semimartingales to have right continuous
paths with left limits. For simplicity, we assume that $\mathscr{F}_{0}$
is trivial and $\mathcal{F}=\mathscr{F}_{T}.$ Denote by ${\mathbb{E}}[\cdot]$
the expectation under $\mathbb{P}$. Conditional expectations with
respect to a sub-$\sigma$ algebra $\mathscr{G}$ of $\mathcal{F}$
are denoted by $\mathbb{E}^{\mathscr{G}}[\cdot].$ Let $\mathscr{B}(\Lambda)$
denote the Borel $\sigma$-algebra of the topological space $\Lambda.$
Let $W=\{W(t)=(W^{1}(t),W^{2}(t),\cdots,W^{d}(t))^{\top}|t\in[0,T]\}$
be a $d$-dimensional standard Brownian motion with respect to its
natural filtration under $\mathbb{P}$. Let $(\Lambda,\mathscr{B}(\Lambda))$
be a measurable space and $\nu$ a finite measure defined on it. Denote
by $\mu$ an integer-valued random measure 
\[
\mu(de,dt)=\{\mu(\omega,de,dt)|\omega\in\Omega\}
\]
on $([0,T]\times\Lambda,\mathscr{B}([0,T])\otimes\mathscr{B}(\Lambda))$
induced by a stationary $\mathscr{F}$-Poisson point process $(p_{t})_{t\geq0}$
on $\Lambda$ with the Lévy measure $\nu.$ Let $\tilde{\mu}(de,dt):={\mu}(de,dt)-\nu(de)dt$
be the compensated Poisson random measure. Suppose that the Brownian
motion $W$ and the random measure $\tilde{\mu}(de,dt)$ are stochastically
independent under $\mathbb{P}$. Without loss of general assumptions,
we assume that the filtration $\mathscr{F}$ is the $\mathbb{P}$-augmentation
of the natural filtration generated by the Brownian motion and the
Poisson random measure.

Let $\mathscr{P}$ be the $\mathscr{F}$-predictable $\sigma$-field
on $\Omega\times[0,T]$ and denote
\[
\mathscr{\tilde{P}}:=\mathscr{P}\otimes\mathscr{B}(\Lambda).
\]
 For a $\mathscr{\tilde{P}}$-measurable function $U$ on $\tilde{\Omega}$$,$
define its integration with respect to $\mu$ (analogously for $\nu\otimes{\rm Leb}$)
by 

\begin{equation}
\int_{0}^{T}\int_{\Lambda}U(s,e)\mu(de,ds)(\omega)=\left\{ \begin{array}{rl}
{\displaystyle \int_{0}^{T}\int_{\Lambda}U(\omega,s,e)\mu(\omega,ds,de),\quad} & \mbox{if finitely defined},\\
+\infty,\qquad\qquad\qquad\quad & \mbox{otherwise}.
\end{array}\right.\label{eq:3.5}
\end{equation}
 The random measure and stochastic integrals can be referred to \cite{IKWA89,protter2005stochastic}
for details. 

\subsection{Introduction on BSREJ}

Denote by $\mathbb{S}^{n}$ the space of all $n\times n$ symmetric
matrices and by $\mathbb{S}_{+}^{n}$ the space of all $n\times n$
nonnegative matrices. Throughout this paper, the following standard
assumptions holds. Suppose that $A,B,C,D,E,F,Q,N$ and $M$ are given
random mappings such that $A:[0,T]\times\Omega\rightarrow\mathbb{R}^{n\times n},B:[0,T]\times\Omega\rightarrow\mathbb{R}^{n\times m};C^{i}:[0,T]\times\Omega\rightarrow\mathbb{R}^{n\times n},D^{i}:[0,T]\times\Omega\rightarrow\mathbb{R}^{n\times m},i=1,2,\cdots,d;E:[0,T]\times\Omega\times\Lambda\rightarrow\mathbb{R}^{n\times n};F:[0,T]\times\Omega\times\Lambda\rightarrow\mathbb{R}^{n\times m};Q:[0,T]\times\Omega\rightarrow\mathbb{R}^{n\times n},N:[0,T]\times\Omega\rightarrow\mathbb{R}^{m\times m};M:\Omega\rightarrow\mathbb{R}^{n\times n}$
satisfies :

\begin{ass}\label{ass:4.1} $A,B,C,D,N$ and $Q$ are uniformly bounded
$\mathscr{F}$-predictable stochastic processes. $E$ and $F$ are
uniformly bounded $\mathscr{\tilde{P}}$-measurable stochastic processes.
$M$ is a uniformly bounded ${\mathscr{F}}_{T}$-measurable random
variable. Moreover, for a.s. a.e. $(t,\omega)\in[0,T]\times\Omega$,
$Q\in\mathbb{S}_{+}^{n}$ and $N\in\mathbb{S}_{+}^{m}$. $M\in\mathbb{S}_{+}^{n}$
for a.e. $\omega\in\Omega$. And $N$ is uniformly positive, i.e.
for a.s. a.e. $(t,\omega)\in[0,T]\times\Omega,$ $N(t)\geq\delta I$
for some positive constant $\delta$.\end{ass} 

For any $(t,K,L,R(\cdot))\in[0,T]\times\mathbb{S}^{n}\times(\mathbb{S}^{n})^{d}\times{\mathcal{M}}^{\nu,2}(\mathbb{S}^{n})$
(see the meaning of the notations in subsection \ref{subsec:Notations}),
define 
\begin{eqnarray}
\mathscr{N}(t,K,R(\cdot)) & := & N(t)+{\displaystyle \sum_{i=1}^{d}(D^{i})^{*}(t)KD^{i}(t)+\int_{\Lambda}F^{*}(t,e)(K+R(e))F(t,e)\nu(de)},\nonumber \\
\mathscr{M}(t,K,L,R(\cdot)) & := & KB(t)+\sum_{i=1}^{d}L^{i}D^{i}(t)+\sum_{i=1}^{d}(C^{i})^{*}(t)KD^{i}(t)\nonumber \\
 &  & +{\displaystyle \int_{\Lambda}\Big[E^{*}(t,e)KF(t,e)+(I+E^{*}(t,e))R(e)F(t,e)\Big]\nu(de),}\\
G(t,K,L,R(\cdot)) & := & A^{*}(t)K+KA(t)+\sum_{i=1}^{d}L^{i}C^{i}(t)+\sum_{i=1}^{d}(C^{i})^{*}(t)L^{i}+\sum_{i=1}^{d}(C^{i})^{*}(t)KC^{i}(t)\nonumber \\
 &  & +{\displaystyle \int_{\Lambda}R(e)E(t,e)\nu(de)+{\displaystyle \int_{\Lambda}E^{*}(t,e)R(e)\nu(de)}}\nonumber \\
 &  & +{\displaystyle \int_{\Lambda}E^{*}(t,e)(K+R(e))E(t,e)\nu(de)}\nonumber \\
 &  & +Q(t)-\mathscr{M}(t,K,L,R(\cdot))\mathscr{N}^{-1}(t,K,R(\cdot))\mathscr{M}^{*}(t,K,L,R(\cdot)),\nonumber 
\end{eqnarray}
where $I$ is $n$-th order identity matrix and $*$ denotes the transpose
of a matrix.

With the notations defined above, we introduce the following backward
stochastic integral-differential equation driven by Brownian motion
$W$ and Poisson random measure $\tilde{\mu}$ : 
\begin{equation}
\left\{ \begin{array}{lll}
dK(t) & = & -G(t,K(t-),L(t),R(t,\cdot))dt+{\displaystyle \sum_{i=1}^{d}L^{i}(t)dW_{t}^{i}+{\displaystyle \int_{\Lambda}R(t,e)\tilde{\mu}(de,dt),}}\\
K(T) & = & M,\quad L(t):=(L^{1}(t),\cdots,L^{d}(t)).
\end{array}\right.\label{eq:Riccati}
\end{equation}
with the unknown triple of stochastic processes $(K,L,R).$ Now we
give the definition of the solution to BSREJ \eqref{eq:Riccati} as
follows.

\begin{defn} \label{def:solution}A triplet of stochastic processes
$(K,L,R)$ valued in $\mathbb{S}^{n}\times(\mathbb{S}^{n})^{d}\times{\cal M}^{\nu,2}(\mathbb{S}^{n})$
with $K$ being $\mathscr{F}$-progressive measurable, $L$ $\mathscr{F}$-predictable
and $R$ $\tilde{\mathscr{P}}$-measurable is called a solution of
BSREJ \eqref{eq:Riccati} if

(i)$\int_{0}^{T}|G(t,K(t-),L(t),R(t))|dt+\int_{0}^{T}\int_{\Lambda}|R(t,e)|^{2}\nu(de)dt+\int_{0}^{T}|L(t)|^{2}dt<\infty,a.s.;$

(ii) $\mathscr{N}(t,K(t-),R(t))$ is positive definite a.s. a.e.; 

(iii) for all $t\in[0,T],$ it a.e. holds that
\begin{equation}
K(t)=M+\int_{t}^{T}G(s,K(s-),L(s),R(s))ds-\int_{t}^{T}\sum_{i=1}^{d}L^{i}(s)dW_{s}^{i}-\int_{t}^{T}\int_{\Lambda}R(s,e)\tilde{\mu}(de,ds).\label{eq:integ-Riccati}
\end{equation}

\end{defn}

This is the so-called BSREJ associated with a linear quadratic optimal
control problem with jumps formulated in Section 2 (See Problem \ref{pro:4.1}).
When the coefficients $A,B,C,D,E,F,Q,N$ are all deterministic, then
$L^{1}=\cdots=L^{d}=R=0$, and the BSREJ \eqref{eq:Riccati} degenerates
to a deterministic Riccati integral-differential equation (see \cite{Wonham1968On}
for the case without jumps). If $D=0$ and $F=0$, i.e. the corresponding
controlled differential system does not contain control in martingale
integration terms, and the second and third unknown variables $(L,R)$
only have a linear structure in the generator $G$. And in this case
the solvability of BSREJ could be covered by the result of Meng \cite{meng2014general}.
Due to that the martingale integration parts of corresponding controlled
system \eqref{eq:4.1} contains control variable, and the system has
non-Markovian structure, the associated BSREJ \eqref{eq:Riccati}
is highly nonlinear with respect to the unknown triple of $(K,L,R)$.

\subsection{Developments of BSRE and Contributions of this Paper}

The study of BSREs had quite a long history. In the case of BSREs
driven by only Brownian motion $W,$ \eqref{eq:Riccati} will reduce
to the following form: {\footnotesize{}
\begin{equation}
\left\{ \begin{array}{ll}
dK(t)= & -\bigg[A^{*}(t)K(t)+K(t)A(t)+\sum_{i=1}^{d}L^{i}(t)C^{i}(t)+\sum_{i=1}^{d}(C^{i})^{*}(t)L^{i}(t)+\sum_{i=1}^{d}(C^{i})^{*}(t)K(t)C^{i}(t)+Q\\
 & -[K(t)B(t)+\sum_{i=1}^{d}L^{i}(t)D^{i}(t)+\sum_{i=1}^{d}(C^{i})^{*}(t)K(t)D^{i}(t)][N(t)+\sum_{i=1}^{d}(D^{i})^{*}(t)K(t)D^{i}(t)]^{-1}\\
 & \quad\cdot[K(t)B(t)+\sum_{i=1}^{d}L^{i}(t)D^{i}(t)+\sum_{i=1}^{d}(C^{i})^{*}(t)K(t)D^{i}(t)]^{*}\bigg]dt+{\displaystyle \sum_{i=1}^{d}L^{i}(t)dW_{t}^{i},}\\
K(T)= & M,\quad L(t):=(L^{1}(t),\cdots,L^{d}(t)).
\end{array}\right.\label{eq:1.5}
\end{equation}
}Historically speaking, the French mathematician Bimut \cite{Bismut1976linear}
firstly proposed the definition of the adapted solution to \eqref{eq:1.5}
, and due to the difficulty of its solvability, it is listed as an
open problem by Peng \cite{peng1999}. Until 2013, Tang \cite{tang2003general}
generally solved this open problem applying the stochastic maximum
principle and using the technique of stochastic flow for the associated
stochastic Hamiltonian system. In 2015, Tang \cite{tang2015dynamic}
gives the second but more comprehensive (seeming much simpler, by
Doob-Meyer decomposition theorem and Dynamic programming principle)
method to solve the general BSREs.

For earlier history on BSRE, we refer to Peng \cite{Peng1992Stochastic},
Tang and Kohlmann \cite{KohlmannTang2001New,KohlmannTang2003Global},
Tang \cite{tang2003general} and the plenary lecture reported by Peng
\cite{peng2010backward} at the ICM in 2010. For the indefinite BSRE,
the reader can be referred to \cite{ChenZhou1998Stochastic,ZhouLi2000Continuous,KohlmannTang2003mini,KohlmannTang2003Multi,QianZhou2013Existence,Du2015Solvability}.

Equation \eqref{eq:Riccati} is very different from equation \eqref{eq:1.5}.
From a direct viewpoint, Equation \eqref{eq:Riccati} is driven by
both a Brownian motion $W$ and an additional compensated Poisson
measure $\tilde{\mu}$. From an essential viewpoint, not only the
first unknown element $K$ and but also the third unknown element
$R$ are included in the nonlinear term $\mathscr{N}(t,K(t-),R(t,\cdot))^{-1}$
in BSREJ \eqref{eq:Riccati}. For the BSRE driven only by a Brownian
motion, the nonlinear term $\mathscr{N}(t,K(t-),R(t,\cdot))^{-1}$
degenerates into $\big[N(t)+D^{i*}(t)K(t)D^{i}(t)\big]^{-1}$ which
is well defined since in that case we can show that $K$ is continuous
and nonnegative. But for the BSREJ \eqref{eq:Riccati}, one only expects
to prove the square integrability of the third unknown element $R$,
but this regularity is difficult to derive the non-negativity of matrix
$\mathscr{N}(t,K(t-),R(t,\cdot))$. How to show $\mathscr{N}(t,K(t-),R(t,\cdot))$
keeping to be positive is key to give the solvability of BSREJ \eqref{eq:Riccati}.

As far as we know, there is very few literature related to BSREJ.
In 2008, under partial information framework, Hu and {Ø}ksendal
\cite{huoksendal2008partial} studied the one-dimensional SLQ problem
with random coefficients and Poisson jumps, where they presented the
state feedback representation of the optimal control by an one-dimensional
BSREJ, but the authors did not discuss the wellposeness of the solution
to BSREJ. \cite{meng2014general} is the first work addressed to the
study of high dimensional SLQ with random coefficients, the author
formally derived BSREJ \eqref{eq:Riccati} and utilized Bellman's
principle of quasi-linearization to solve a special form of BSREJ
\eqref{eq:Riccati} in which the generator $G$ only linearly depends
on $L$ and $R$. Li et al \cite{LiWuYu2018indefinite} used so-called
relax compensator to describe indefinite BSREJ and investigated the
solvability BSREJ in some special cases. 

The contributions of our paper is to establish the solvability of
the general BSREJ \eqref{eq:Riccati}. Adapting the method proposed
by Tang \cite{tang2015dynamic}, with the help of control problem
and dynamic programming principle, we use the value function and Doob-Meyer
decomposition to construct the triple process $(K(t),L(t),R(t,\cdot))$
and later show it is nothing but the solution of BSREJ \eqref{eq:Riccati}.
Conversely, we also could utilize the solution of BSREJ \eqref{eq:Riccati}
to depict the optimal control in a feedback form. 

One advantage of above method is to avoid the proof of the positive
definiteness of the matrix process $\mathscr{N}$ at the beginning.
In our approach, we show not only the positive definiteness of of
$\mathscr{N}$, but also that of $\int_{\Lambda}F^{*}(t,e)(K(t-)+R(t,e))F(t,e)\nu(de)$.
The proof is based on an observation that: $\int_{\Lambda}R(t,e)\mu(de,\{t\})$
is nothing but the jump measure of $K(t)$. Hence the value $\int_{\Lambda}F^{*}(t,e)(K(t-)+R(t,e))F\mu(de,\{t\})$
vanishes except at the jump time, then it coincides with
\begin{equation}
\int_{\Lambda}F^{*}(t,e)K(t)F(t,e)\mu(de,\{t\})\label{eq:intro-+}
\end{equation}
 since the jump $\Delta K_{t}=K(t)-K(t-)=R(t,\Delta p_{t})$, where
$\Delta p_{t}$ is the jump of underlying Poisson process. Obviously
\eqref{eq:intro-+} is positive once the positive definiteness of
$\mathscr{N}$ obtained.

The inverse flow of the controlled stochastic differential equation
on interval $[0,T]$ is a key technique in Tang's method in \cite{tang2015dynamic}
to give the representation of the BSREJ. In some literature about
stochastic differential with jumps \cite{fujiwarakunita@1985stochastic,Kunita2004stochastic,tanghou2002optimal,chentang2015semi},
the authors give a technical condition to guarantee its inverse flow
exists on $[0,T]$ (using the notation of SDE \eqref{eq:4.1})
\begin{equation}
I+E(t,e)\geq\delta I,\quad{\rm a.e.a.s.},\:{\rm for}\:{\rm some}\:\delta>0.\label{eq:not-neces-condit}
\end{equation}
But this condition is not necessary for the LQ control problem. In
our approach, to overcome the difficulty brought from the absence
of condition \eqref{eq:not-neces-condit}, we deal with SDE \eqref{eq:4.1}
in every stochastic sub-interval between every two adjacent jumping
time $(\!(\tau_{i},\tau_{i+1})\!)$, on which SDE \eqref{eq:linear-SDE}
has continuous trajectory solution and subsequently inverse flow without
the help of condition \ref{eq:not-neces-condit}. Then we use the
semi-martingale property of $K$ to integrate all the sub-intervals
to obtain the representation of BSREJ on the whole interval $[0,T]$. 

The rest of this article is organized as follows. In Section \ref{sec:prel},
we introduce some useful notations, preliminary results and the SLQ
problem with jumps. In Section \ref{sec:DPP-semimartingle}, we list
the preliminary results and the controlled SLQ problem. Section \ref{sec:DPP-semimartingle}
gives some basic properties of the value function $V$, and also the
semimartingale property of $V$ by dynamic programming principle.
In Section \ref{sec:Exist}, with the help of results in Section \ref{sec:DPP-semimartingle}
we show the existence of BSREJ \eqref{eq:Riccati}. In Section \ref{sec:Verification},
we show the verification theorem which gives the uniqueness of the
solution for BSREJ, and use the solution of BSREJ to describe the
optimal control and valuation of the SLQ problem. 

\section{Preliminary Results and SLQ Problem\label{sec:prel}}

\subsection{Notations\label{subsec:Notations}}

Let $H$ be a Hilbert space. The inner product in $H$ is denoted
by $\langle\cdot,\cdot\rangle,$ and the norm in $H$ is denoted by
$|\cdot|_{H}$ or $|\cdot|$ if there is no danger of confusion. Let
$p\geq1.$ Let $\mathscr{T}$ denote the totality of all $\mathscr{F}$-stopping
times taking values in $[0,T].$ Define $\mathscr{T}_{\tau}:=\{\gamma\in\mathscr{T}:\gamma\geq\tau,\,\mathbb{P}{\rm -a.s.}\}$
for $\tau\in\mathscr{T}.$ Given $\tau\in\mathscr{T}$ and $\gamma\in\mathscr{F}_{\tau}$,
the following spaces will be frequently used in this paper: 
\begin{enumerate}
\item[$\bullet$] ${\mathcal{S}}_{\mathscr{F}}^{p}(\tau,\gamma;H)$: the set of all
$H$-valued $\mathscr{F}$-adapted right continuous left limit (RCLL)
processes $f\triangleq\{f(t,\omega),t\in[\![\tau,\gamma]\!]\}$ such
that $\|f\|_{{S}_{\mathscr{F}}^{p}(\tau,\gamma;H)}:=\bigg\{{\mathbb{E}}\bigg[\sup_{\tau\leq t\leq\gamma}|f(t)|_{H}^{p}\bigg]\bigg\}^{\frac{1}{p}}<\infty$;
\item[$\bullet$] ${\mathcal{M}}_{\mathscr{F}}^{p}(\tau,\gamma;H)$: the set of all
$H$-valued $\mathscr{F}$-progressively measurable processes $f\triangleq\{f(t,\omega),t\in[\![\tau,\gamma]\!]\}$
such that $\|f\|_{{\mathcal{M}}_{\mathscr{F}}^{p}(\tau,\gamma;H)}:=\bigg\{\mathbb{E}{\displaystyle \bigg[\int_{\tau}^{\gamma}|f(t)|_{H}^{p}dt\bigg]\bigg\}^{\frac{1}{p}}<\infty}$;
\item[$\bullet$] ${\cal M}_{\mathscr{F}}^{2,p}(\tau,\gamma;H)$: the set of all $H$-valued
$\mathscr{F}$-progressively measurable processes $f\triangleq\{f(t,\omega),t\in[\![\tau,\gamma]\!]\}$
such that $\|f\|_{{\cal M}_{\mathscr{F}}^{2,p}(\tau,\gamma;H)}:=\bigg\{{\mathbb{E}}\bigg[{\displaystyle \int_{\tau}^{\gamma}|f(t)|_{H}^{2}dt\bigg]^{\frac{p}{2}}\bigg\}^{\frac{1}{p}}<\infty}$;
\item[$\bullet$] ${\cal M}^{\nu,2}(H):$ the set of all H-valued measurable functions
$r\triangleq\{r(e),e\in\Lambda\}$ defined on the measure space $({\Lambda},\mathscr{B}(\Lambda),\nu)$
such that $\|r\|_{{\cal M}^{\nu,2}(H)}:=\sqrt{{\displaystyle \int_{\Lambda}|r(e)|_{H}^{2}\nu(de)}}<~\infty$;
\item[$\bullet$] ${\cal M}_{\mathscr{F}}^{\nu,2,p}(\tau,\gamma;H):$ the set of all
$H$-valued $\mathscr{\tilde{P}}$-measurable processes $r\triangleq\{r(t,\omega,e),(t,e)\in[\![\tau,\gamma]\!]\times\Lambda\}$
such that $\|r\|_{\mathcal{M}_{\mathscr{F}}^{\nu,2,p}(\tau,\gamma;H)}:=\bigg\{\mathbb{E}\bigg[{\displaystyle \int_{\tau}^{\gamma}\int_{\Lambda}}|r(t,e)|_{H}^{2}\nu(de)dt\bigg]^{\frac{p}{2}}\bigg\}^{\frac{1}{p}}<~\infty$;
\item[$\bullet$] ${\cal M}_{\mathscr{F}}^{\nu,p}(\tau,\gamma;H):$ the set of all
$H$-valued $\mathscr{\tilde{P}}$-measurable processes $r\triangleq\{r(t,\omega,e),(t,e)\in[\![\tau,\gamma]\!]\times\Lambda\}$
such that $\|r\|_{\mathcal{M}_{\mathscr{F}}^{\nu,p}(\tau,\gamma;H)}:=\bigg\{\mathbb{E}{\displaystyle \bigg[\int_{\tau}^{\gamma}\int_{\Lambda}|r(t,e)|_{H}^{p}\nu(de)dt\bigg]\bigg\}^{\frac{1}{p}}<~\infty}$;
\item[$\bullet$] $L^{p}(\Omega,{\mathscr{G}},\mathbb{P};H):$ the set of all $H$-valued
${\mathscr{G}}$-measurable random variable $\xi$ defined on $(\Omega,\mathcal{F},P)$
such that $\|\xi\|_{L^{p}(\Omega,{\mathscr{G}},P;H)}:=\{\mathbb{E}[|\xi|_{H}^{p}]\}^{\frac{1}{p}}$
where $\mathscr{G}$ is a subalgebra of $\mathcal{F}$. 
\end{enumerate}
In the following we recall a classical theorem for the \emph{essential
infimum} of a family of nonnegative random variables in a probability
space (see, e.g. Karatzas and Shreve \cite[Appendix A]{KaraShreve1998method}).

\begin{lem} \label{lem:2.1} Let $\mathscr{X}$ be a family of nonnegative
integrable random variables defined on a probability space $(\Omega,\mathcal{F},\mathbb{P}).$
Then there exists an $\mathcal{F}$-measurable random variable $X^{*}$
such that \\
 1. for all $X\in\mathscr{X},$ $X\geq X^{*}$ a.s.;\\
 2. if $Y$ is a random variable satisfying $X\geq Y$ a.s. for all
$X\in\mathscr{X},$ then $X^{*}\geq Y$ a.s.\\
 This random variable, which is unique a.s., is called the essential
infimum of $\mathscr{X},$ and is denoted by $\essinf\mathscr{X}$
or $\essinf_{X\in\mathscr{X}}X$. Furthermore, if $\mathscr{X}$ is
closed under pairwise minimum (i.e. $X,Y\in\mathscr{X}$ implies $X\wedge Y\in\mathscr{X}$),
then there exists a nondecreasing sequence $\{Z_{n}\}_{n\in\mathbb{N}}$
of random variables in $\mathscr{X}$ such that $X^{*}=\lim_{n\rightarrow\infty}Z_{n}$
a.s. Moreover, for any sub-algebra $\mathscr{G}$ of $\mathcal{F},$
the $\mathscr{G}$-conditional expectation is interchangeable with
the essential infimum: 
\[
\mathbb{E}[\essinf_{X\in\mathscr{X}}X|\mathscr{G}]=\essinf_{X\in\mathscr{X}}\mathbb{E}[X|\mathscr{G}].
\]

\end{lem}

\subsection{Some Basic Definition and Results on $\mathscr{T}$-System}

For any $\tau_{1},\tau_{2}\in\mathscr{T},$ with $\tau_{1}\leq\tau_{2}$
almost surely and $\mathbb{P}(\tau_{1}<\tau_{2})>0,$ let 
\[
\mathscr{T}[\tau_{1},\tau_{2}]:=\{\tau\in\mathscr{T}|\tau_{1}\le\tau\le\tau_{2}~\mathbb{P}{\rm -a.s.}\}.
\]
 The following classical result of aggregation of supmartingale system
could be found in \cite{ElKaroui1981Les}.

\begin{defn} \label{defn:T-system}A family of random variables $\mathscr{K}:=\{\mathscr{K}(\tau),\tau\in\mathscr{T}\}$
indexed by $\mathscr{T}$ is said to be $\mathscr{T}$-system if it
satisfies\\
 1. for all $\tau\in\mathscr{T}$, $\mathscr{K}(\tau)$ is $\mathscr{F}_{\tau}$-measurable
random variable;\\
 2. for all $\tau_{1},\tau_{2}\in\mathscr{T},$ $\mathscr{K}(\tau_{1})=\mathscr{K}(\tau_{2})$
a.s. on $\{\tau_{1}=\tau_{2}\}~{\rm for}~\tau_{1},\tau_{2}\in\mathscr{T}.$
\end{defn}

\begin{defn} We call a $\mathscr{T}$-system $\{\mathscr{K}(\tau),\tau\in\mathscr{T}\}$
a submartingale system if the following two properties hold: \\
(i) $\mathscr{K}(\tau)$ is integrable for any $\tau\in\mathscr{T};$
\\
(ii) $\mathbb{E}^{\mathscr{F}_{\tau_{1}}}[\mathscr{K}(\tau_{2})]\geq\mathscr{K}(\tau_{1}),$
$\mathbb{P}$-a.s., for all $\tau_{1}\in\mathscr{T},\tau_{2}\in\mathscr{T}_{\tau_{1}}.$
\\
We call $\mathscr{T}$-system $\mathscr{K}:=\{\mathscr{K}(\tau),\tau\in\mathscr{T}\}$
is said to be a supermartingale system if $-\mathscr{K}$ is a submartingale
system, and call it a martingale if it is both a $\mathscr{T}$-supermartingale
and a $\mathscr{T}$-submartingale system. \end{defn}

\begin{defn}\label{defn: submart} A $\mathscr{T}$-system $\{\mathscr{K}(\tau),\tau\in\mathscr{T}\}$
is called right-(resp., left-) continuous along times in expectation
(RCE (resp., LCE)) if for any sequences of stopping times $(\tau_{n})_{n\in\mathbb{N}}$
such that $\tau_{n}\searrow\tau$ a.s.(resp., $\tau_{n}\nearrow\tau$
), one has $\mathbb{E}[\mathscr{K}(\tau)]=\lim_{n\longrightarrow\infty}\mathbb{E}[\mathscr{K}(\tau_{n})].$
\end{defn}

\begin{defn}\label{defn: rce} We call that an process $X=\{X(t),t\in[0,T]\}$
aggregates the $\mathscr{T}$-system $\{\mathscr{K}(\tau),\tau\in\mathscr{T}\},$
if for any $\tau\in\mathscr{T},$ it holds $X(\tau)=\mathscr{K}(\tau),$
$\mathbb{P}$-a.s. \end{defn}

The following result could be found in \cite[subsection 2.14 on p.112]{ElKaroui1981Les},
or adapted from \cite[Theorem 3.13 in Chapter 1]{KaratzsShreve1991Brownian}. 

\begin{prop}\label{prop: aggregation} Let a $\mathscr{T}$-system
$\{\mathscr{K}(\tau),\tau\in\mathscr{T}\}$ be a supermartingale system
which is RCE and such that $\mathscr{K}(0)<\infty$. There then exists
a RCLL adapted process denoted by $\{K(t)\}_{t\in[0,T]}$ which aggregates
$\mathscr{T}$-system $\{\mathscr{K}(\tau),\tau\in\mathscr{T}\}.$
\end{prop}

\begin{proof}

Consider a supermartingale process $(\mathscr{K}(t))_{0\leq t\leq T}$,
by Theorem 3.13 in \cite[Chapter 1]{KaratzsShreve1991Brownian}, it
has a RCLL modification $K(t):=\lim_{s\searrow t,s\in\mathbb{Q}}\mathscr{K}(s)$.
For any stopping time $\tau$, define $\tau_{n}(\omega):=\frac{i}{2^{n}}$,
if $\tau(\omega)\in(\frac{i-1}{2^{n}},\frac{i}{2^{n}}]$ for some
integer $i>0$. It is easy to see that $\mathscr{K}(\tau_{n})=K(\tau_{n})$.
Then by REC of $\mathscr{K}$ and uniform convergence of $\{K(\tau_{n})\}$
(see Remark 3.12 in \cite[Chapter 1]{KaratzsShreve1991Brownian}),
passing $n$ to infinity, we have $\mathscr{K}(\tau)=K(\tau)$ a.e.
Thus $\{K(t)\}_{t\in[0,T]}$ aggregates $\mathscr{T}$-system $\{\mathscr{K}(\tau),\tau\in\mathscr{T}\}$. 

\end{proof}

For future purposes, we shall consider the \textquotedbl conditional\textquotedbl{}
extension of $\mathscr{T}$-system. More precisely, for a family of
random variables ${\mathscr{K}}:=\{\mathscr{K}(\sigma),\sigma\in\mathscr{T}_{\tau}\}$
indexed by $\mathscr{T}_{\tau}$, it is called a $\mathscr{T}_{\tau}$-system
if it satisfies\\
 1. for all $\sigma\in\mathscr{T}_{\tau}$, $\mathscr{K}(\sigma)$
is $\mathscr{F}_{\sigma}$-measurable random variable.\\
 2. for all $\sigma_{1},\sigma_{2}\in\mathscr{T}_{\tau},$ $\mathscr{K}(\sigma_{1})=\mathscr{K}(\sigma_{2})$
a.s.~ on $\{\sigma_{1}=\sigma_{2}\}~{\rm for}~\sigma_{1},\sigma_{2}\in\mathscr{T}_{\tau}.$\\
 Naturally, Definitions \ref{defn: submart} and \ref{defn: rce}
can be adapted for the $\mathscr{T}_{\tau}$-system. Given a $\mathscr{T}_{\tau}$-system
$\mathscr{K}$, one can extend it to be a $\mathscr{T}$-system, still
denoted by ${\mathscr{K}}$, in the following way: 
\[
\mathscr{K}(\sigma):=\mathscr{K}(\sigma)\chi_{\{\sigma\ge\tau\}}+\mathbb{E}[\mathscr{K}(\tau)\chi_{\{\sigma<\tau\}}|\mathscr{F}_{\sigma}]\chi_{\{\sigma<\tau\}}.
\]
If the original $\mathscr{T}_{\tau}$-system ${\mathscr{K}}$ is a
submartingale (resp. supermartingale) system, then the extension is
also a submartingale (resp. supermartingale) system. Moreover, the
RCE (or LCE) property holds for the extension. Hence, according to
Proposition \ref{prop: aggregation}, if ${\mathscr{K}}$ is a supermartingale
$\mathscr{T}_{\tau}$-system which is RCE and $E[\mathscr{K}(\tau)]<+\infty$,
then there exists a RCLL adapted process $K$ defined on the random
interval $[\![\tau,T]\!]$ which aggregates $\mathscr{K}$, i.e.,
for any $\sigma\in\mathscr{T}_{\tau}$, 
\[
K(\sigma)=\mathscr{K}(\sigma),\mathbb{P}-a.s..
\]

\subsection{Preliminary Results for Liner SDE with Jumps}

Let $p\geq2.$ For any $(\tau,\xi)\in\mathscr{T}\times L^{p}(\Omega,{\mathscr{F}_{\tau}},\mathbb{P};\mathbb{R}^{n}),$
consider the following linear SDE with jumps

\begin{equation}
\left\{ \begin{array}{lll}
dX(t) & = & [A(t)X(t-)+f(t)]dt+{\displaystyle \sum_{i=1}^{d}[C^{i}(t)X(t-)+g^{i}(t)]dW^{i}(t)}\\
 &  & +{\displaystyle \int_{\Lambda}[E(t,e)X({t-})+h(t,e)]\tilde{\mu}(de,dt),\tau\leq t\leq T,}\\
x(\tau) & = & \xi,
\end{array}\right.\label{eq:linear-SDE}
\end{equation}
where the coefficients satisfy the following basic assumption: \begin{ass}\label{ass:1.1}
The matrix-valued processes $A:[0,T]\times\Omega\rightarrow\mathbb{R}^{n\times n},B:[0,T]\times\Omega\rightarrow\mathbb{R}^{n\times m};C^{i}:[0,T]\times\Omega\rightarrow\mathbb{R}^{n\times n},i=1,2,\cdots,d$
are uniformly bounded and $\mathscr{F}$-predictable. The matrix process
$E:[0,T]\times\Omega\times\Lambda\rightarrow\mathbb{R}^{n\times n}$
is uniformly bounded and $\mathscr{\tilde{P}}$-measurable. The stochastic
processes $f(\cdot),g^{i}(\cdot)$ belong to ${\cal M}_{\mathscr{F}}^{2,p}(0,T;\mathbb{R}^{n})$
and $h(\cdot,\cdot)$ belongs to ${\cal M}_{\mathscr{F}}^{\nu,p}(0,T;\mathbb{R}^{n}).$
\end{ass}

The following classical estimate could be found in lots of literature
(see \cite{protter2005stochastic,LiPeng2009stochastic}), the proof
based on the It\^o formula, Gronwall's inequality and BDG inequality
is standard.

\begin{lem} \label{lem:3.1} Let Assumptions \ref{ass:1.1} be satisfied.
Then the SDE \eqref{eq:linear-SDE} has a unique strong solution $X(\cdot)\in{\cal S}_{{\cal F}}^{p}(\tau,T;\mathbb{R}^{n})$
and there is a constant $C_{p}>0$ such that for any stopping time
$\tau<T$,{\footnotesize{}
\begin{equation}
\begin{split}\mathbb{E}^{\mathscr{F}_{\tau}}\bigg[\sup_{\tau\leq t\leq T}|X(t)|^{p}\bigg]\leq C_{p}\mathbb{E}^{\mathscr{F}_{\tau}}\bigg[|\xi|^{p}+\bigg(\int_{\tau}^{T}|f(t)|^{2}dt\bigg)^{\frac{p}{2}}+\bigg(\int_{\tau}^{T}\sum_{i=1}^{d}|g^{i}(t)|^{2}dt\bigg)^{\frac{p}{2}}+\int_{\tau}^{T}\int_{\Lambda}|h(t,e)|^{p}\nu(de)dt\bigg].\end{split}
\label{eq:2.4}
\end{equation}
}\end{lem} 

\subsection{Formulation on SLQ Problem}

In this section, we formulate the SLQ problem with jumps. We first
give the following definition of admissible control.

\begin{defn} Let $\tau\in\mathscr{T}.$ An $\mathscr{F}$-predictable
process $u(\cdot)$ is said to be an admissible control on the random
interval $[\![\tau,T]\!],$ if $u(\cdot)\in{\cal M}_{\mathscr{F}}^{2}(\tau,T;\mathbb{R}^{m}).$
The set of all admissible control is denoted by $\mathscr{U}_{\tau}$
\end{defn} For any given admissible control $u(\cdot)\in\mathscr{U}_{0}$,
consider the following controlled linear SDE with jumps: 
\begin{equation}
\left\{ \begin{array}{lll}
dX(t) & = & [A(t)X(t-)+B(t)u(t)]dt+{\displaystyle \sum_{i=1}^{d}[C^{i}(t)X(t-)+D^{i}(t)u(t)]dW^{i}(t)}\\
 &  & +{\displaystyle \int_{\Lambda}[E(t,e)X({t-})+F(t,e)u(t)]\tilde{\mu}(de,dt),}\\
X(0) & = & x
\end{array}\right.\label{eq:4.1}
\end{equation}
with the cost functional 
\begin{equation}
\begin{split}J(u(\cdot);0,x):=\mathbb{E}\bigg[\langle MX(T),X(T)\rangle+\int_{0}^{T}\big(\langle Q(t)X(t),X(t)\rangle+\langle N(t)u(t),u(t)\rangle\big)dt\bigg].\end{split}
\label{eq:2.2}
\end{equation}
Here $A,B,C,D,E,F,Q,N$ and $M$ are given random mappings such that
$A:[0,T]\times\Omega\rightarrow\mathbb{R}^{n\times n};B:[0,T]\times\Omega\rightarrow\mathbb{R}^{n\times m};C^{i}:[0,T]\times\Omega\rightarrow\mathbb{R}^{n\times n},D^{i}:[0,T]\times\Omega\rightarrow\mathbb{R}^{n\times m},i=1,2,\cdots,d;E:[0,T]\times\Omega\times\Lambda\rightarrow\mathbb{R}^{n\times n};F:[0,T]\times\Omega\times\Lambda\rightarrow\mathbb{R}^{n\times m};Q:[0,T]\times\Omega\rightarrow\mathbb{R}^{n\times n},N:[0,T]\times\Omega\rightarrow\mathbb{R}^{m\times m};M:\Omega\rightarrow\mathbb{R}^{n\times n}$
satisfying Assumption \ref{ass:4.1}.

By Lemma \ref{lem:3.1}, for any $u(\cdot)\in{\mathscr{U}}_{0},$
it follows that the SDE \eqref{eq:4.1} admits a unique strong solution
in the space ${\cal S}_{\mathscr{F}}^{2}(0,T;\mathbb{R}^{n})$, denoted
by $X^{0,x;u(\cdot)}(\cdot)$. We call $X(\cdot)\triangleq X^{0,x;u(\cdot)}(\cdot)$
the state process corresponding to the control process $u(\cdot)$
and call $(u(\cdot);X(\cdot))$ the admissible pair. Furthermore,
Assumption \ref{ass:4.1} and the a priori estimate \eqref{eq:2.4}
imply that 
\[
|J(u(\cdot);0,x)|<\infty.
\]

Then our SLQ problem can be stated as follows.

\begin{pro}\label{pro:4.1} Find an admissible control process ${\bar{u}}(\cdot)\in{\mathscr{U}}_{0}$
such that 
\begin{equation}
J({\bar{u}}(\cdot);0,x)=\inf_{u(\cdot)\in{\mathscr{U}}_{0}}J(u(\cdot);0,x).\label{eq:b8}
\end{equation}
\end{pro} 

The admissible control ${\bar{u}}(\cdot)$ satisfying \eqref{eq:b8}
is called an optimal control process of Problem \ref{pro:4.1}. Correspondingly,
the state process ${\bar{X}}(\cdot)$ associated with ${\bar{u}}(\cdot)$
is called an optimal state process and $({\bar{u}}(\cdot);{\bar{X}}(\cdot))$
is called an optimal pair of Problem \ref{pro:4.1}.

\section{\label{sec:DPP-semimartingle}Dynamical Programming Principle and
the Semimartingale Property of the Value Process}

\subsection{Initial-Data-Parameterized SLQ Problem}

This subsection is devoted to introducing the initial-data-parameterized
SLQ Problem. For simplicity, we define the random function 
\[
f(t,x,u):=\langle Q(t)x,x\rangle+\langle N(t)u,u\rangle,\quad\forall(t,x,u)\in[0,T]\times\mathbb{R}^{n}\times\mathbb{R}^{m}.
\]
Fixed initial data $(\tau,\xi)\in\mathscr{T}\times L^{2}(\Omega,{\mathscr{F}_{\tau}},\mathbb{P};\mathbb{R}^{n}),$
for any given admissible control $u(\cdot)\in\mathscr{U}_{\tau},$
denote by $X^{\tau,\xi;u}$ the solution of following state equation

\begin{equation}
\left\{ \begin{array}{lll}
dX(t) & = & [A(t)X(t-)+B(t)u(t)]dt+{\displaystyle \sum_{i=1}^{d}[C^{i}(t)X(t-)+D^{i}(t)u(t)]dW^{i}(t)}\\
 &  & +{\displaystyle \int_{\Lambda}[E(t,e)X({t-})+F(t,e)u(t)]\tilde{\mu}(de,dt),}\\
X(\tau) & = & \xi.
\end{array}\right.\label{eq:3.2}
\end{equation}
The cost functional is defined as the following conditional expectation:
\begin{equation}
J(u(\cdot);\tau,\xi):=\mathbb{E}^{\mathscr{F}_{\tau}}{\displaystyle \bigg[{\displaystyle \int_{\tau}^{T}f(s,X^{\tau,\xi;u(\cdot)}(s),u(s))ds+\langle MX^{\tau,\xi;u(\cdot)}(T),X^{\tau,\xi;u(\cdot)}(T)\rangle\bigg].}}\label{eq:3.3}
\end{equation}
Then the corresponding initial-data-parameterized SLQ Problem is stated
as follows :

\begin{pro}\label{pro:3.1} Find an admissible control process ${\bar{u}}(\cdot)\in\mathscr{U}_{\tau}$
such that 
\begin{equation}
J({\bar{u}}(\cdot);\tau,\xi)=\essinf_{u(\cdot)\in\mathscr{U}_{\tau}}J(u(\cdot);\tau,\xi).\label{eq:ocp}
\end{equation}
\end{pro} We also denote the above optimal control problem by Problem
$\mathscr{P}_{\tau,\xi}$ to stress the dependence on the parameter
$(\tau,\xi).$ Clearly, for any initial data $(\tau,\xi)\in\mathscr{T}\times L^{2}(\Omega,{\mathscr{F}_{\tau}},\mathbb{P};\mathbb{R}^{n})$
and admissible control $u(\cdot)\in\mathscr{U}_{\tau},$ the state
equation \eqref{eq:3.2} has a unique strong solution $X(\cdot)\equiv X^{\tau,\xi;u(\cdot)}$
and \eqref{eq:ocp} is well-defined. Furthermore, we can define the
following conditional minimal value system

\begin{equation}
\mathscr{V}(\tau,\xi):=\essinf_{u(\cdot)\in\mathscr{U}_{\tau}}J(u(\cdot);\tau,\xi).\label{eq:3.4}
\end{equation}
It is obvious that $\mathscr{V}(\tau,\xi)$ is $\mathscr{F}_{\tau}$-measurable
random variable for any $(\tau,\xi)\in\mathscr{T}\times L^{2}(\Omega,{\mathscr{F}_{\tau}},\mathbb{P};\mathbb{R}^{n})$.
The random variable $\mathscr{V}(\tau,\xi)$ will play an important
role in the dynamic programming principle method to obtain the existence
of the solution of the BSREJ \eqref{eq:Riccati}. 

The following two results Proposition \ref{prop:quadr} and Theorem
\ref{thm: dpp} are needed in our approach. The description and their
proofs are more or less standard in the context of SLQ problem. We
just give a sketch of the proof in the case of jumps since it is similar
to that in the case of Brownian motion. We suggest the reader to visit
Sections 2 and 3 in \cite{tang2015dynamic} for full details. 

\begin{prop} \label{prop:quadr} Let Assumption \ref{ass:4.1} hold. 

(i) There is a positive constant $\lambda$ such that for any $(\tau,\xi)\in\mathscr{T}\times L^{2}(\Omega,{\mathscr{F}_{\tau}},\mathbb{P};\mathbb{R}^{n})$,
it has 
\begin{equation}
0\leq\mathscr{V}(\tau,\xi)\leq J(0;\tau,\xi)\leq\lambda|\xi|^{2}.
\end{equation}

(ii) For any given initial data $(\tau,\xi)\in\mathscr{T}\times L^{2}(\Omega,{\mathscr{F}_{\tau}},\mathbb{P};\mathbb{R}^{n}),$
Problem $\mathscr{P}_{\tau,\xi}$ has a unique optimal control $\bar{u}(\cdot)\in\mathscr{U}_{\tau}$,
i.e. 
\[
\mathscr{V}(\tau,\xi)=J(\bar{u}(\cdot);\tau,\xi),\,\ensuremath{\mathbb{P}}\text{-a.s.}
\]

(iii) The value functional $\mathscr{V}(\tau,\xi)$ is quadratic with
respect to $\xi.$ Moreover, there is an $\mathbb{S}_{+}^{n}$-valued
family $\mathscr{K}:=\{\mathscr{K}(\tau),\tau\in{\mathscr{T}}\}$
such that $\mathscr{K}(\tau)$ is essentially bounded for any $\tau\in\mathscr{T}$
and $\xi\in L^{2}(\Omega,\mathscr{F}_{\tau},\mathbb{P};\mathbb{R}^{n})$
\begin{equation}
\mathscr{\text{\ensuremath{\mathbb{V}}}}(\tau,\xi)=\langle\mathscr{K}(\tau)\xi,\xi\rangle.\label{eq:V-K-relat}
\end{equation}

(iv) For each $x\in\mathbb{R}^{n},$ define the family 
\[
\mathscr{V}_{x}:=\{\mathscr{V}(\tau,x),\tau\in\mathscr{T}\}.
\]
 Then it is a $\mathscr{T}$-system. Moreover, the family $\mathscr{K}=\{\mathscr{K}(\tau),\tau\in{\mathscr{T}}\}$
is also a $\mathscr{T}$-system.

\end{prop}

\begin{proof} 

(i) Noting Assumption \ref{ass:4.1} and \eqref{eq:3.4}, it is sufficient
to show $J(0;\tau,\xi)\leq\lambda|\xi|^{2}.$ In fact, from the a
priori estimate \eqref{eq:2.4}, we get that

\[
\begin{split}J(0;\tau,\xi) & \leq C\mathbb{E}^{\mathscr{F}_{\tau}}{\displaystyle \bigg[{\displaystyle \int_{\tau}^{T}|X^{\tau,\xi;0}(t)|^{2}dt+|X^{\tau,\xi;0}(T)|^{2}\bigg]}}\\
 & \leq C\mathbb{E}^{\mathscr{F}_{\tau}}\bigg[\sup_{\tau\leq t\leq T}|X^{\tau,\xi;0}(t)|^{2}\bigg]\\
 & \leq C|\xi|^{2}.
\end{split}
\]

(ii) Let $(\tau,\xi)\in\mathscr{T}\times L^{2}(\Omega,{\mathscr{F}_{\tau}},\mathbb{P};\mathbb{R}^{n})$.
For any $u_{1}(\cdot),u_{2}(\cdot)\in\mathscr{U}_{\tau}$, define
\[
\hat{u}(\cdot):=u_{1}(\cdot)\chi_{\{J(u_{1}(\cdot);\tau,\xi)\leq J(u_{2}(\cdot);\tau,\xi)\}}+u_{2}(\cdot)\chi_{\{J(u_{1}(\cdot);\tau,\xi)>J(u_{2}(\cdot);\tau,\xi)\}}.
\]
 Then $X^{\tau,\xi;\hat{u}(\cdot)}=X^{\tau,\xi;u_{1}(\cdot)}\chi_{\{J(u_{1}(\cdot);\tau,\xi)\leq J(u_{2}(\cdot);\tau,\xi)\}}+X^{\tau,\xi;u_{2}(\cdot)}\chi_{\{J(u_{1}(\cdot);\tau,\xi)>J(u_{2}(\cdot);\tau,\xi)\}}.$
Hence 
\begin{align*}
J(\hat{u}(\cdot);\tau,\xi)= & J(u_{1}(\cdot);\tau,\xi)\chi_{\{J(u_{1}(\cdot);\tau,\xi)\leq J(u_{2}(\cdot);\tau,\xi)\}}+J(u_{2}(\cdot);\tau,\xi)\chi_{\{J(u_{1}(\cdot);\tau,\xi)>J(u_{2}(\cdot);\tau,\xi)\}}\\
= & \min\{J(u_{1}(\cdot);\tau,\xi),J(u_{2}(\cdot);\tau,\xi)\}.
\end{align*}
That is $\{J(u(\cdot);\tau,\xi):u(\cdot)\in\mathscr{U}_{\tau}\}$
is closed under pairwise minimum. By Lemma \ref{lem:2.1}, there is
a sequence $\{u_{k}(\cdot)\}_{k=1}^{\infty}\subset\mathscr{U}_{\tau}$,
such that 
\begin{equation}
J(u_{k}(\cdot);\tau,\xi)\searrow\mathscr{V}(\tau,\xi),\quad{\rm as}\,k\to\infty.\label{eq:min-seq}
\end{equation}
 By the parallelogram equality,
\begin{align*}
2J\big(\frac{1}{2}(u_{k}(\cdot)-u_{l}(\cdot));\tau,\xi\big)+2\mathscr{V}(\tau,\xi)\leq & 2J\big(\frac{1}{2}(u_{k}(\cdot)-u_{l}(\cdot));\tau,\xi\big)+2J\big(\frac{1}{2}(u_{k}(\cdot)+u_{l}(\cdot));\tau,\xi\big)\\
= & J(u_{k}(\cdot);\tau,\xi)+J(u_{l}(\cdot);\tau,\xi).
\end{align*}
Let $k,l\to\infty$ in the following inequality, 
\[
0\leq2J\big(\frac{1}{2}(u_{k}(\cdot)-u_{l}(\cdot));\tau,\xi\big)\leq J(u_{k}(\cdot),\tau,\xi)+J(u_{l}(\cdot),\tau,\xi)-2\mathscr{V}(\tau,\xi)\to0,
\]
 which means $\{u_{k}(\cdot)\}_{k=1}^{\infty}$ is Cauchy sequence
in ${\cal M}_{\mathscr{F}}^{2}(\tau,T;\mathbb{R}^{m})$. And it is
easy to check that $\bar{u}(\cdot):=\lim_{k\to\infty}u_{k}(\cdot)$
is the unique optimal control for problem $\mathscr{P}_{\tau,\xi}$. 

(iii) One can show that (see \cite{Faurre1968sur} or \cite[Lemma 3.2]{tang2015dynamic}),
for any real number $\eta>0$, $x,y\in\mathbb{R}^{n}$,
\begin{align*}
\mathscr{V}(\tau,\eta x) & =\eta^{2}\mathscr{V}(\tau,x),\\
\mathscr{V}(\tau,x+y)+\mathscr{V}(\tau,x-y) & =2\mathscr{V}(\tau,x)+2\mathscr{V}(\tau,y).
\end{align*}
So $\mathscr{V}(\tau,x)$ is a quadratic form. Let 
\begin{equation}
\begin{split}\mathscr{K}(\tau)=\frac{1}{4}(\mathscr{V}(\tau,e_{i}+e_{j})-\mathscr{V}(\tau,e_{i}-e_{j}))_{{i,j=1}}^{n},\end{split}
\label{eq:3.9}
\end{equation}
 then we have \eqref{eq:V-K-relat}.

(iv) Verifying Definition \ref{defn:T-system} directly, we shall
prove that $\mathscr{V}_{x}$ is $\mathscr{T}$-system and and consequently
so does $\mathscr{K}$. \end{proof}

\subsection{Dynamical Programming Principle and the Semimartingale Property}

The following result is the dynamical programming principle for Problem
$\mathscr{P}_{\tau,\xi}.$

\begin{thm}\label{thm: dpp} Let Assumption \ref{ass:4.1} hold.
(i) For $\tau\in\mathscr{T},\sigma\in\mathscr{T}_{\tau},$ and $\xi\in L^{2}(\Omega,\mathscr{F}_{\tau},\ensuremath{\mathbb{P}};\mathbb{R}^{n}),$
\begin{equation}
\begin{split}\mathscr{V}(\tau,\xi)=\essinf_{u(\cdot)\in{\mathscr{U}}_{\tau}}\mathbb{E}^{\mathscr{F}_{\tau}}\bigg[\int_{\tau}^{\sigma}f(s,X^{\tau,\xi;u(\cdot)}(s),u(s))ds+\mathscr{V}(\sigma,X^{\tau,\xi;u(\cdot)}(\sigma))\bigg].\end{split}
\label{eq:dpp-1}
\end{equation}
And it holds that 
\begin{equation}
\begin{split}\mathscr{V}(\tau,\xi)=\mathbb{E}^{\mathscr{F}_{\tau}}\bigg[\int_{\tau}^{\sigma}f(s,X^{\tau,\xi;\bar{u}(\cdot)}(s),\bar{u}(s))ds+\mathscr{V}(\sigma,X^{\tau,\xi;\bar{u}(\cdot)}(\sigma))\bigg]\end{split}
\label{eq:dpp-2}
\end{equation}
for the optimal control $\bar{u}(\cdot)\in\mathscr{U}_{\tau}$ of
Problem $\mathscr{P}_{\tau,\xi}$.\\
(ii) For any $\tau\in\mathscr{T}$ and $(x,u(\cdot))\in\mathbb{R}^{n}\times\mathscr{U}_{\tau},$
the family $\mathscr{J}^{\tau,x,u(\cdot)}:=\{\mathscr{J}^{\tau,x,u(\cdot)}(\sigma),\sigma\in\mathscr{T}_{\tau}\}$
is a $\mathscr{T}$-submartingale, where 
\begin{equation}
{\mathscr{J}}^{\tau,x,u(\cdot)}(\sigma):=\mathscr{V}(\sigma,X^{\tau,x;u(\cdot)}(\sigma))+\int_{\tau}^{\sigma}f(r,X^{\tau,x;u(\cdot)}(r),u(r))dr,\quad\sigma\in\mathscr{T}_{\tau};\label{eq:dpp-3}
\end{equation}
 And the family $\mathscr{J}^{\tau,x,\bar{u}(\cdot)}$ is a $\mathscr{T}$-martingale
for the optimal control $\bar{u}(\cdot)\in\mathscr{U}_{\tau}$ of
problem $\mathscr{P}_{\tau,x}$. Besides,

\[
{\mathscr{J}}^{\tau,x,u(\cdot)}(\sigma)=\essinf_{v(\cdot)\in\mathscr{U}_{\sigma}^{u(\cdot)}}\mathbb{E}^{\mathscr{F}_{\sigma}}\bigg[\int_{\tau}^{T}f(r,X^{\tau,x;v(\cdot)}(r),v(r))+\langle MX^{\tau,x;v(\cdot)}(T),X^{\tau,x;v(\cdot)}(T)\rangle\bigg],\:u\in\mathscr{U}_{\tau},
\]
 where
\[
\mathscr{U}_{\sigma}^{u(\cdot)}:=\big\{ v(\cdot)\in\mathscr{U}_{\tau}|v(\cdot)=u(\cdot)~{\rm on}~[\![\tau,\sigma]\!]\big\}.
\]
\\
(iii) If $\bar{u}(\cdot)\in{\mathscr{U}}_{\tau}$ such that $\mathscr{J}^{\tau,x,\bar{u}(\cdot)}$
is a $\mathscr{T}$-martingale, then $\bar{u}(\cdot)$ is optimal
for Problem $\mathscr{P}_{\tau,x}.$

\end{thm}

\begin{proof} 

(i) Similar as \eqref{eq:min-seq}, there is a minimizing sequence
$\{v_{m}(\cdot)\}\subset\mathscr{U}_{\sigma}$ of Problem $\mathscr{P}_{\sigma,X^{\tau,\xi;u(\cdot)}(\sigma)}$
such that, then we have for any $v(\cdot)\in\mathscr{U}_{\sigma}$,
\begin{align*}
\mathbb{E}^{\mathscr{F}_{\tau}}\Big[J(v(\cdot);\sigma,X^{\tau,\xi;u(\cdot)}(\sigma))\Big]\geq & \mathbb{E}^{\mathscr{F}_{\tau}}\Big[\mathscr{V}(\sigma,X^{\tau,\xi;u(\cdot)}(\sigma))\Big]\\
= & \mathbb{E}^{\mathscr{F}_{\tau}}\Big[\inf_{m}J(v_{m}(\cdot);\sigma,X^{\tau,\xi;u(\cdot)}(\sigma))\Big]\\
= & \essinf_{m}\mathbb{E}^{\mathscr{F}_{\tau}}\Big[J(v_{m}(\cdot);\sigma,X^{\tau,\xi;u(\cdot)}(\sigma))\Big]\\
\geq & \essinf_{v(\cdot)\in\mathscr{U}_{\sigma}}\mathbb{E}^{\mathscr{F}_{\tau}}\Big[J(v(\cdot);\sigma,X^{\tau,\xi;u(\cdot)}(\sigma))\Big].
\end{align*}
 Taking $\essinf_{v(\cdot)\in\mathscr{U}_{\sigma}}$ on the left hand
side of above inequality, then the inequalities turn to equalities.
We have 
\[
\essinf_{v(\cdot)\in\mathscr{U}_{\sigma}}\mathbb{E}^{\mathscr{F}_{\tau}}\Big[J(v(\cdot);\sigma,X^{\tau,\xi;u(\cdot)}(\sigma))\Big]=\mathbb{E}^{\mathscr{F}_{\tau}}\Big[\mathscr{V}(\sigma,X^{\tau,\xi;u(\cdot)}(\sigma))\Big].
\]
 Furthermore for any $u(\cdot)\in\mathscr{U}_{\tau}$, 
\begin{align*}
 & \mathbb{E}^{\mathscr{F}_{\tau}}\Big[\int_{\tau}^{\sigma}f(s,X^{\tau,\xi;u(\cdot)}(s),u(s))ds+\mathscr{V}(\sigma,X^{\tau,\xi;u(\cdot)}(\sigma))\Big]\\
= & \essinf_{v(\cdot)\in\mathscr{U}_{\sigma}}\mathbb{E}^{\mathscr{F}_{\tau}}\Big[\int_{\tau}^{\sigma}f(s,X^{\tau,\xi;u(\cdot)}(s),u(s))ds+J(v(\cdot);\sigma,X^{\tau,\xi;u(\cdot)}(\sigma))\Big]\\
= & \essinf_{v(\cdot)\in\mathscr{U}_{\sigma}}\mathbb{E}^{\mathscr{F}_{\tau}}\Big[J(u(\cdot)\otimes v(\cdot);\tau,\xi)\Big],
\end{align*}
where $u(\cdot)\otimes v(\cdot)=u(\cdot)$ on $[\![\tau,\sigma]\!]$,
and $u(\cdot)\otimes v(\cdot)=v(\cdot)$ on $[\![\sigma,T]\!]$. \eqref{eq:dpp-1}
is the result of taking $\essinf_{u(\cdot)\in\mathscr{U}_{\tau}}$
on both sides of above equality.

If $\bar{u}(\cdot)\in\mathscr{U}_{\tau}$ is the optimal control for
$\mathscr{P}_{\tau,\xi}$, then its restriction $\bar{u}\big|_{[\![\sigma,T]\!]}(\cdot)$
is the optimal control for $\mathscr{P}_{\tau,X^{\tau,\xi;\bar{u}(\cdot)}(\sigma)}$.
Then \eqref{eq:dpp-2} follows. Then assertion (i) holds.

In view of (i), it is easy to check that (ii) and (iii) hold.

\end{proof}

\begin{lem} Let Assumptions \ref{ass:4.1} be satisfied. Then for
each $x\in\mathbb{R}^{n},$ the $\mathscr{T}$-systems $\mathscr{V}_{x}$
and $\mathscr{K}=\{\mathscr{K}(\tau),\tau\in{\mathscr{T}}\}$ are
RCE. \end{lem}

\begin{proof} For any $\tau\in\mathscr{T}_{0}$, $\tau_{m}\in\mathscr{T}_{\tau}$
satisfying that $\tau_{m}\searrow\tau$ a.s. as $m\to\infty$. By
(i) of Theorem \ref{thm: dpp}, for the optimal control $\bar{u}(\cdot)$
of Problem $\mathscr{P}_{\tau,x}$,
\begin{equation}
\mathscr{V}(\tau,x)=\mathbb{E}^{\mathscr{F}_{\tau}}\bigg[\int_{\tau}^{\tau_{m}}f(s,X^{\tau,x;\bar{u}(\cdot)}(s),\bar{u}(s))ds+\mathscr{V}(\tau_{m},X^{\tau,x;\bar{u}(\cdot)}(\tau_{m}))\bigg].\label{eq:v-tau}
\end{equation}
Since $\mathscr{K}$ is uniformly bounded,
\begin{align*}
 & \mathbb{E}^{\mathscr{F}_{\tau}}\text{\ensuremath{\bigg[}}\Big|\mathscr{V}(\tau_{m},X^{\tau,x;\bar{u}(\cdot)}(\tau_{m}))-\mathscr{V}(\tau_{m},x)\Big|\bigg]\\
= & \mathbb{E}^{\mathscr{F}_{\tau}}\bigg[\Big|\big\langle\mathscr{K}(\tau_{m})X^{\tau,x;\bar{u}(\cdot)}(\tau_{m}),X^{\tau,x;\bar{u}(\cdot)}(\tau_{m})\big\rangle-\big\langle\mathscr{K}(\tau_{m})x,x\big\rangle\Big|\bigg]\\
\leq & \lambda\bigg(\mathbb{E}^{\mathscr{F}_{\tau}}\Big[|x|+|X^{\tau,x;\bar{u}(\cdot)}(\tau_{m})|\Big]^{2}\bigg)^{\frac{1}{2}}\bigg(\mathbb{E}^{\mathscr{F}_{\tau}}\big[X^{\tau,x;\bar{u}(\cdot)}(\tau_{m})-x\big]^{2}\bigg)^{\frac{1}{2}},
\end{align*}
and 
\[
\mathbb{E}^{\mathscr{F}_{\tau}}\bigg[\int_{\tau}^{\tau_{m}}f(s,X^{\tau,x;\bar{u}(\cdot)}(s),\bar{u}(s))ds\bigg]\leq C\mathbb{E}^{\mathscr{F}_{\tau}}\bigg[\int_{\tau}^{\tau_{m}}\big(\big|X^{\tau,x;\bar{u}(\cdot)}(s)\big|^{2}+|\bar{u}(s)|^{2}\big)ds\bigg],
\]
Then by \eqref{eq:v-tau}, the estimate \eqref{eq:2.4} and the dominate
control theorem, we get 
\begin{align*}
 & \mathbb{E}^{\mathscr{F}_{\tau}}\bigg[\big|\mathscr{V}(\tau,x)-\mathscr{V}(\tau_{m},x)\big|\bigg]\\
\leq & \mathbb{E}^{\mathscr{F}_{\tau}}\Big[\int_{\tau}^{\tau_{m}}|f(s,X^{\tau,x;\bar{u}}(s),\bar{u}(s))ds|+\Big|\mathscr{V}(\tau_{m},X^{\tau,x;\bar{u}}(\tau_{m}))-\mathscr{V}(\tau_{m},x)\Big|\Big]\\
\to & 0,\qquad{\rm as}\:m\to\infty.
\end{align*}
 This is the RCE of $\mathscr{V}_{x}$. The RCE of $\mathscr{K}$
is a direct inference of that of $\mathscr{V}_{x}$ and \eqref{eq:V-K-relat}.\end{proof}

\begin{thm} \label{thm:DM-decomposition} Let Assumptions \ref{ass:4.1}
be satisfied.\\
 (i) For any $\tau\in\mathscr{T}$ and $(x,u(\cdot))\in\mathbb{R}^{n}\times\mathscr{U}_{\tau},$
the $\mathscr{T}_{\tau}$-system $\mathscr{J}^{\tau,x,u(\cdot)}$
is RCE and aggregated by a RCLL $\mathscr{F}$-submartingale denoted
by $\{\mathbb{J}^{\tau,x,u(\cdot)}(t),t\in[\![\tau,T]\!]\}$. For
the optimal control $\bar{u}(\cdot)\in\mathscr{U}_{\tau}$ of Problem
$\mathscr{P}_{\tau,x},$ the corresponding $\mathscr{T}_{\tau}$-system
$\mathscr{J}^{\tau,x,\bar{u}(\cdot)}$ is aggregated by a RCLL $\mathscr{F}$-martingale
denoted by $\{\mathbb{J}^{\tau,x,\bar{u}(\cdot)}(t),t\in[\![\tau,T]\!]\}$.\\
 (ii) The $\mathscr{T}$-system $\{\mathscr{K}(\tau),\tau\in\mathscr{T}\}$
is RCE and aggregated by a RCLL process denoted by $\{K(t),t\in[0,T]\}.$
$K$ is essentially bounded and $\mathbb{S}_{+}^{n}$-valued. We have
for any $t\in[0,T]$
\begin{equation}
\begin{split}K(t)=K(0)-\int_{0}^{t}dk(s)+\sum_{i=1}^{d}\int_{0}^{t}L^{i}(s)dW_{s}^{i}+\int_{0}^{t}\int_{\Lambda}R(s,e)d\tilde{\mu}(de,ds),\quad K(T)=M,\end{split}
\label{eq:3.16}
\end{equation}
where $k$ is an $\mathbb{S}^{n}$-valued predictable process of bounded
variation, $L^{i}$ an $\mathbb{S}^{n}$-valued predictable process
and $R$ a $\tilde{\mathscr{P}}$-measurable process. \\
 (iii) The condition minimal value system $\mathscr{V}_{x}$ for $x\in\mathbb{R}^{n}$
is aggregated by the following RCLL semimartingale 
\[
V(t,x):=\langle K(t)x,x\rangle,\quad t\in[0,T].
\]
\end{thm}

\begin{proof} In view of \eqref{eq:dpp-3}, the REC of family $\mathscr{J}^{\tau,x,u(\cdot)}$
comes from that of the $\mathscr{T}$-system $\mathscr{V}_{x}$ and
the a.s. right continuity of maps $t\mapsto X_{t}^{\tau,x,u(\cdot)}$
and $t\mapsto\int_{\tau}^{t}f(s,X(s),u(s))ds$. Using Proposition
\ref{prop: aggregation}, we prove the first part of assertion (i).
From the second part of Theorem \ref{thm: dpp}, we see that $\mathscr{J}^{\tau,x,\bar{u}(\cdot)}$
is a $\mathscr{F}$-martingale.\bigskip{}

Now we begin to show the assertion (ii). Denote by $\tau_{k}$ the
$n$-th jump time of the Poisson point process. Recall that $e_{i}$
is the unit column vector whose $i$-th component is the number $1$
for $i=1,\cdots,n.$ We see that for $x=e_{i},e_{i}+e_{j},e_{i}-e_{j}$
with $i,j=1,\cdots,n$, the process $\mathbb{J}^{k,x}(t):={\mathscr{J}}^{\tau_{k}\wedge T,x,0}(t),t\in[\![\tau_{k}\wedge T,T]\!]$
is a right-continuous submartingale and $\mathbb{J}^{k,x}(\cdot)-\mathbb{J}^{k,x}(\tau_{k}\wedge T)$
is of \textit{class D}. Hence by Doob-Meyer decomposition (see \cite[Theorem 11 in Section III.3]{protter2005stochastic}),
it could be decomposed to an increasing, predictable process and a
uniformly integrable martingale. Consider an $\mathbb{S}^{n}$-valued
$\mathscr{T}_{\tau_{k}\wedge T}$-system $\Gamma_{k}:=\{\Gamma_{k}(\tau),\tau\in\mathscr{T}_{\tau_{k}\wedge T}\}$
defined as follows: 
\begin{equation}
\Gamma_{k}(\tau):=\frac{1}{4}\Big(\mathscr{J}^{\tau_{k}\wedge T,e_{i}+e_{j},0}(\tau)-\mathscr{J}^{\tau_{k}\wedge T,e_{i}-e_{j},0}(\tau)\Big)_{1\leq i,j\leq n},\quad\tau\in\mathscr{T}_{\tau_{k}\wedge T}.\label{eq:semimartiggle-2}
\end{equation}
In view of $X^{\tau_{k}\wedge T,e_{i}\pm e_{j},0}(\tau)=X^{\tau_{k}\wedge T,e_{i},0}(\tau)\pm X^{\tau_{k}\wedge T,e_{j},0}(\tau)$,
\eqref{eq:V-K-relat} and the proof of \eqref{eq:3.9}, we have 
\begin{align*}
 & \big(\mathscr{V}(\tau,X^{\tau_{k}\wedge\tau,e_{i}+e_{j},0}(\tau))-\mathscr{V}(\tau,X^{\tau_{k}\wedge\tau,e_{i}-e_{j},0}(\tau))\big)_{1\leq i,j\le n}\\
= & \big(X^{\tau_{k}\wedge\tau,e_{1},0}(\tau),\ldots,X^{\tau_{k}\wedge\tau,e_{n},0}(\tau)\big)^{*}\mathscr{K}(\tau)\big(X^{\tau_{k}\wedge\tau,e_{1},0}(\tau),\ldots,X^{\tau_{k}\wedge\tau,e_{n},0}(\tau)\big).
\end{align*}
This together with \eqref{eq:dpp-3} and \eqref{eq:semimartiggle-2}
yields
\[
\Gamma_{k}(\tau)=\Phi_{k}^{*}(\tau)\mathscr{K}(\tau)\Phi_{k}(\tau)+\int_{\tau_{k}\wedge T}^{\tau}\Phi_{k}^{*}(r)Q(r)\Phi_{k}(r)dr,
\]
where $\Phi_{k}(t)$ is the solution of the following linear SDE:
\begin{equation}
\left\{ \begin{array}{l}
d\Phi(t)=A(t)\Phi(t-)dt+\sum_{i=1}^{d}C^{i}(t)\Phi(t-)dW^{i}(t)+\int_{\Lambda}E(t,e)\Phi({t-})\tilde{\mu}(de,dt),\\
\Phi(\tau_{k}\wedge T)=I,\qquad\qquad t\in(\negthickspace(\tau_{k}\wedge T,\tau_{k+1}\wedge T)\negthickspace).
\end{array}\right.\label{eq:sm-3-0}
\end{equation}
The $\mathscr{T}_{\tau_{k}\wedge T}-$ system $\Gamma_{k}$ is aggregated
by the following process still denoted by $\{\Gamma_{k}(t),t\in[\![\tau_{k}\wedge T,T)\negthickspace)\}:$
\[
\Gamma_{k}(t)=:\frac{1}{4}(\mathbb{J}^{k,e_{i}+e_{j}}(t)-\mathbb{J}^{k,e_{i}-e_{j}}(t))_{1\leq i,j\leq n},t\in[\![\tau_{k}\wedge T,T)\negthickspace),
\]
which is a right-continuous semimartingale with predictable of bounded
variational part. We see that $\Phi_{k}(t)$ is reversible for $t\in(\negmedspace(\tau_{k}\wedge T,\tau_{k+1}\wedge T)\negmedspace)$\textcolor{brown}{{}
}and its inverse $\Psi_{k}(t):=\Phi_{k}^{-1}(t)$ satisfying 
\begin{equation}
\left\{ \begin{array}{l}
d\Psi_{k}(t)=\Psi_{k}(t-)\bigg[-A(t)+C^{2}(t)+\int_{\Lambda}E(t,e)\nu(de)\bigg]dt-\sum_{i=1}^{d}\Psi_{k}(t-)C^{i}(t)dW^{i}(t),\\
\Psi_{k}(\tau_{k}\wedge T)=I,\qquad\qquad t\in(\negthickspace(\tau_{k}\wedge T,\tau_{k+1}\wedge T)\negthickspace).
\end{array}\right.\label{eq:sm-3}
\end{equation}
 It is obvious that $\Psi_{k}(t)$ is continuous at $[\![\tau_{k}\wedge T,\tau_{k+1}\wedge T)\negthickspace)$
and has left-limit at $\tau_{k+1}\wedge T$. Define 
\[
K_{k}(t):=\Psi_{k}^{*}(t)\Gamma_{k}(t)\Psi_{k}(t)-\Psi_{k}^{*}(t)\int_{\tau_{k}\wedge T}^{t}\Phi_{k}^{*}(s)Q(s)\Phi_{k}(s)ds\Psi_{K}(t),\,t\in[\![\tau_{k}\wedge T,\tau_{k+1}\wedge T)\negthickspace).
\]
It is continuous on $[\![\tau_{k}\wedge T,\tau_{k+1}\wedge T)\negmedspace)$
and has left-limit at $\tau_{k+1}\wedge T$. By It\^o formula, $K_{k}$
is a semimartingale, i.e. 
\[
K_{k}(t)=K_{k}(\tau_{k}\wedge T)+\tilde{M}_{k}(t)+\tilde{A}_{k}(t),\quad t\in[\![\tau_{k}\wedge T,\tau_{k+1}\wedge T)\negthickspace)
\]
where $\tilde{M}_{k}$ with $\tilde{M}_{k}(\tau_{k}\wedge T)=0$ is
a local martingale and $\tilde{A}$ with $\tilde{A}(\tau_{k}\wedge T)=0$
a predictable process with finite variation. We see that $\mathscr{K}(\tau)=K_{k}(\tau)$
for $\tau_{k}\wedge T\le\tau<\tau_{k+1}\wedge T$. Thus $\mathscr{K}$
is aggregated by the process 
\begin{align*}
K(t):= & \sum_{k=0}^{\infty}K_{k}(t)\chi_{\{\tau_{k}\wedge T\le t<\tau_{k+1}\wedge T\}}\\
= & \Big(\sum_{\tau_{k+1}\leq t}\tilde{M}_{k}\big((\tau_{k+1}\wedge T)-\big)+\tilde{M}_{i}(t)\Big)+\Big(\sum_{\tau_{k+1}\leq t}\tilde{A}_{k}\big((\tau_{k+1}\wedge T)-\big)+\tilde{A}_{i}(t)\Big)\\
 & \quad+\sum_{\tau_{k}\leq t,k>1}\big(K_{k}(\tau_{k}\wedge T)-K_{k-1}((\tau_{k}\wedge T)-)\big),
\end{align*}
where $i$ is the maximal integer with $\tau_{i}\leq t$. It is easy
to observe that the first term of the right hand of above equality
is a continuous martingale, the second term is continuous bounded
variational process, and the third term is a pure jump process. By
localizing method, it is easy to know the first part of last term
is a local martingale, second part a finite variational predictable
process. According to $K_{k}$ is uniformly bounded, Theorem 35 in
\cite[Section III.7]{protter2005stochastic} yields the pure jump
process $\sum_{\tau_{k}\leq t,k>1}\big(K_{k}(\tau_{k}\wedge T)-K_{k-1}((\tau_{k}\wedge T)-)\big)$
is a \textit{special semimartingale}. Thus $K$ could be canonically
decomposed into the sum of an $\mathscr{F}$-predictable process $k_{t}$
with finite variation and an $\mathscr{F}$-martingale process on
the whole time interval $[0,T]$. By martingale representation theorem
(see \cite[Section 5, Chapter IV]{protter2005stochastic} or \cite[Lemma 2.3]{TangLi94}
for a easier version), we know that $K$ can be written as (\ref{eq:3.16}). 

At last, the assertion (iii) is just a result of \eqref{eq:V-K-relat}.
Thus we finish the proof. \end{proof}

\begin{rmk}

In \cite{tang2003general,tang2015dynamic}, the inverse flow of the
controlled SDE is the key technique to show $K_{t}$ to be the fist
part of the triple processes solution of BSRE. And in \cite{tang2003general},
the author pays lots of calculus to prove that the inverse flow of
the solution$X$ for SDE associated with the corresponding optimal
control exists on the whole time interval. For the SDE with jump,
its inverse flow may not exists on whole time $[0,T]$ without additional
condition, e.g., 
\begin{equation}
I+E\geq\delta I,\quad{\rm a.e.,a.s.}\label{eq:addi-cond}
\end{equation}
 However, condition \eqref{eq:addi-cond} is not necessary for the
original control problem. So we insist on not introducing the condition
\eqref{eq:addi-cond} in the formulation of our BSREJ. 

We observe that in the form of optimal feedback (see \eqref{eq:feedBackCont}),
$K$ is independent of the state of the controlled equation, which
hint us to represent $K$ by different state process in different
time interval. Hence to overcome the difficulty of absence of \eqref{eq:addi-cond},
we can piece-wisely represent $K$ by the inverse flow on sub-interval
between two adjacent jump time, on which the SDE \eqref{eq:sm-3-0}
has continuous trajectories hence an inverse flow. After that we integrated
the representation of $K$ from piece-wise to whole process on $[0,T]$
by the semimartingale property.\end{rmk}

\section{\label{sec:Exist}Existence of Solutions to BSREJ}

This section is devoted to showing that $(K,L,R)$ given by Theorem
\ref{thm:DM-decomposition} is nothing other than the solution of
BSREJ \eqref{eq:Riccati}, and to giving their estimates. Thus we
establish the existence of solution for BSREJ \eqref{eq:Riccati}.

\begin{thm} Let Assumptions \ref{ass:4.1} be satisfied. Then $(K,L,R)$
given by Theorem \ref{thm:DM-decomposition} satisfies BSREJ (\ref{eq:Riccati}).
And there is a deterministic constant $C$ such that the following
estimate holds: 
\begin{equation}
\mathbb{E}\bigg(\int_{0}^{T}\sum_{i=1}^{d}\big|L^{i}(t)\big|^{2}ds\bigg)+\mathbb{E}\bigg(\int_{0}^{T}\int_{\Lambda}\big|R(t,e)\big|^{2}\nu(de)dt\bigg)\leq C.\label{eq:exist-0}
\end{equation}
 Hence ${\displaystyle \int_{0}}^{\cdot}L^{i}(s)dW_{s}^{i}+{\displaystyle \int}_{0}^{\cdot}\int_{E}R(e,s)\tilde{\mu}(de,ds)$
is a BMO martingale. Moreover ${\displaystyle \int}_{\Lambda}F^{*}(t,e)(K(t-)+R(t,e))F(t,e)\nu(de)$
is nonnegative for almost all $t$, $P$-a.s.. \end{thm}

\begin{proof} \textit{Firstly, we show that $(K,L,R)$ satisfies
satisfies \eqref{eq:integ-Riccati} a.e.a.s.} Define the functional
\begin{align*}
 & \mathbb{F}(t,x,u,K(t),L(t),R(t,\cdot))\\
:= & 2\langle K(t)x,A(t)x+B(t)u\rangle+2\sum_{i=1}^{d}\langle L^{i}(t)x,C^{i}(t)x+D^{i}(t)u\rangle+\sum_{i=1}^{d}\langle K(t)(C^{i}(t)x+D^{i}(t)u),C^{i}(t)x+D^{i}(t)u\rangle\\
 & +2{\displaystyle \int_{\Lambda}\Big\langle R(t,e)x,E(t,e)x+F(t,e)u\Big\rangle\nu(de)}+{\displaystyle \int_{\Lambda}\Big\langle\big(K(t)+R(t,e)\big)(E(t,e)x+F(t,e)u),E(t,e)x+F(t,e)u\Big\rangle\nu(de).}
\end{align*}

For $\tau\in\mathscr{T}$, $\sigma\in\mathscr{T}_{\tau}$ and $u(\cdot)\in\mathscr{U}_{\tau},$
applying Itô formula to $V(t,X^{\tau,x;u(\cdot)}(t))=\langle K(t)X^{\tau,x;u(\cdot)}(t),X^{\tau,x;u(\cdot)}(t)\rangle$,
we get

\begin{equation}
\begin{split} & V(\sigma,X^{\tau,x;u(\cdot)}(\sigma)))\\
= & V(\tau,x)-\int_{\tau}^{\sigma}\langle dk(t)X,X\rangle+\int_{\tau}^{\sigma}\mathbb{F}(t,X,u(t),K(t-),L(t),R(t,\cdot))dt\\
 & +\sum_{i=1}^{d}\int_{\tau}^{\sigma}\bigg[\langle L^{i}(t)X,X\rangle+2\langle K(t-)X,C^{i}(t)X+D^{i}(t)u(t)\rangle\bigg]dW^{i}(t)\\
 & +\int_{\tau}^{\sigma}\int_{\Lambda}\langle R(t,e)(X+E(t,e)X+F(t,e)u(t)),X+E(t,e)X+F(t,e)u(t)\rangle\tilde{\mu}(dt,de)\\
 & +\int_{\tau}^{\sigma}\int_{\Lambda}\langle K(t-)(E(t,e)X+F(t,e)u(t)),E(t,e)X+F(t,e)u(t)\rangle\tilde{\mu}(dt,de)\\
 & +2\int_{\tau}^{\sigma}\int_{\Lambda}\langle K(t-)X,E(t,e)X+F(t,e)u(t)\rangle\tilde{\mu}(dt,de),
\end{split}
\label{eq:exist-3}
\end{equation}
where $X$ is short for $X^{\tau,x;u(\cdot)}(t-)$. Taking conditional
expectation with $\mathscr{F}_{\tau}$ on both sides of the above
relation and noting the fact that the conditional expectation of the
stochastic integrals w.r.t. the Brownian motion $W$ and the Poisson
random measure $\tilde{\mu}$ vanishes by the localization with the
stopping time, we obtain 
\begin{equation}
\begin{split} & \mathbb{E}^{\mathscr{F}_{\tau}}[V(\sigma,X^{\tau,x;u(\cdot)}(\sigma))]\\
= & V(\tau,x)+\mathbb{E}^{\mathscr{F}_{\tau}}\bigg[\int_{\tau}^{\sigma}\Pi(t,X^{\tau,x;u(\cdot)}(t-))dt\bigg]-\mathbb{E}^{\mathscr{F}_{\tau}}\bigg[\int_{\tau}^{\sigma}f(t,X^{\tau,x;u(\cdot)}(t-),u(t))dt\bigg],
\end{split}
\label{eq:4.10}
\end{equation}
where 
\begin{equation}
\begin{split} & \Pi(dt;\tau,x,u(\cdot))\\
:= & -\langle dk(t)X^{\tau,x;u(\cdot)}(t-),X^{\tau,x;u(\cdot)}(t-)\rangle+\mathbb{F}(t,X^{\tau,x;u(\cdot)}(t-),u(t),K(t-),L(t),R(t,\cdot))dt\\
 & +f(t,X^{\tau,x;u(\cdot)}(t-),u(t))dt.
\end{split}
\label{eq:4.10-1}
\end{equation}
This implies that 
\begin{equation}
\begin{split} & \mathbb{E}^{\mathscr{F}_{\tau}}\bigg[\int_{\tau}^{\sigma}\Pi(dt;\tau,x,u(\cdot))dt\bigg]\\
= & \mathbb{E}^{\mathscr{F}_{\tau}}[V(\sigma,X^{\tau,x;u(\cdot)}(\sigma))]+\mathbb{E}^{\mathscr{F}_{\tau}}\bigg[\int_{\tau}^{\sigma}l(t,X^{\tau,x;u(\cdot)}(t-),u(t))dt\bigg]-V(\tau,x).
\end{split}
\label{eq:4.10}
\end{equation}
From the dynamic programming principle, we have 
\begin{equation}
\begin{split} & \mathop{\text{ess.}\inf}_{u(\cdot)\in\mathscr{U}_{\tau}}\mathbb{E}^{\mathscr{F}_{\tau}}\bigg[\int_{\tau}^{\sigma}\Pi(dt;\tau,x,u(\cdot))dt\bigg]\\
= & \mathop{\text{ess.}\inf}_{u(\cdot)\in\mathscr{U}_{\tau}}\bigg\{\mathbb{E}^{\mathscr{F}_{\tau}}[V(\sigma,X^{\tau,x;u(\cdot)}(\sigma))]+\mathbb{E}^{\mathscr{F}_{\tau}}\bigg[\int_{\tau}^{\sigma}f(t,X^{\tau,x;u(\cdot)}(t-),u(t))dt\bigg]\bigg\}-V(\tau,x)\\
= & V(\tau,x)-V(\tau,x)\\
= & 0.
\end{split}
\label{eq:4.10}
\end{equation}
 Choose $\tau_{k}$ as the $k$-th jump time of the Poisson point
process. This implies that the measure $\Pi(ds;\tau_{k}\wedge T,x,u(\cdot))dxdP$
is nonnegative on $\{(t,x,\omega):t\in(\tau_{k}(\omega)\wedge T,T],x\in\mathbb{R}^{n},\omega\in\Omega\}$
for any $u(\cdot)\in\mathscr{U}_{\tau}.$ Therefore, for any essentially
bounded nonnegative predictable field $\eta$ defined on $[0,T]\times\mathbb{R}^{n}\times\Omega$,
we have 
\[
\begin{split}\mathbb{E}\int_{\tau_{k}\wedge T}^{\tau_{k+1}\wedge T}\int_{\mathbb{R}^{n}}\eta(s,X^{\tau_{k}\wedge T,x;u(\cdot)}(s))\det(\Phi_{k}(s))\Pi(ds;\tau_{k}\wedge T,x,u(\cdot))\geq0,\quad\forall u(\cdot)\in\mathscr{U}_{0}\end{split}
\]
with $\Phi_{k}(s)$ being the Jacobian matrix of flow transformation
$x\longrightarrow X^{\tau_{k}\wedge T,x;u(\cdot)}(s)$ for any $u(\cdot)\in U_{\tau_{k}\wedge T}$.
Note that before the next jump time $\tau_{k+1}$, $\Phi(s)$ is inversible,
i.e., $\det(\Phi(s))>0$ $\mathbb{P}$-a.s. Via a transformation of
state variable $x$, we have

\[
\begin{split}\mathbb{E}\int_{\tau_{k}\wedge T}^{\tau_{k+1}\wedge T}\int_{\mathbb{R}^{n}}\eta(s,x)\Pi(ds;\tau_{k}\wedge T,Y^{\tau_{k}\wedge T,x;u(\cdot)}(s),u(\cdot))\geq0,\quad\forall u(\cdot)\in\mathscr{U}_{0},\end{split}
\]
where $Y^{\tau_{k}\wedge T,x;u(\cdot)}(s)$ is the inverse of the
flow $x\longrightarrow X^{\tau_{k}\wedge T,x;u(\cdot)}(s)$ for $\tau_{k}\wedge T\leq s<\tau_{k+1}\wedge T$.
Incorporating $\Pi(ds;0,\cdot,u(\cdot))\geq0$ with the inverse flow
$Y^{\tau_{k}\wedge T,x;u(\cdot)}(s),x\in\mathbb{R}^{n}$ we have 
\[
\begin{split}0\leq & \Pi(dt;\tau_{k}\wedge T,Y^{\tau_{k}\wedge T,x;u(\cdot)}(t);u(\cdot))\\
= & -\langle dk(t)x,x\rangle+\mathbb{F}(t,x,u(t),K(t-),L(t),R(t,\cdot))dt\\
 & \qquad+f(t,x,u(t))dt
\end{split}
\]
on $\{(t,\omega):t\in(\tau_{k}(\omega)\wedge T,\tau_{k+1}(\omega)\wedge T),\omega\in\Omega\}$.
In a similar way, we have for a.e. a.s. $(t,\omega)\in[0,T]\times\Omega,$

\[
\begin{split}0= & \Pi(dt;\tau_{k}\wedge T,Y^{\tau_{k}\wedge T,x;\bar{u}(\cdot)}(t),\bar{u}(\tau_{k}\wedge T,Y^{\tau_{k}\wedge T,x;\bar{u}(\cdot)}(t)))\\
= & -\langle dk(t)x,x\rangle+\mathbb{F}(t,x,\bar{u}(\tau_{k}\wedge T,Y^{\tau_{k}\wedge T,x;\bar{u}(\cdot)}(t)),K(t-),L(t),R(t,\cdot))dt\\
 & +f(t,x,\bar{u}(\tau_{k}\wedge T,Y^{\tau_{k}\wedge T,x;\bar{u}(\cdot)}(t)))dt.
\end{split}
\]
Therefore, we have 
\[
\begin{split}\langle dk(t)x,x\rangle=\min_{v\in\mathbb{R}^{n}}\big[\mathbb{F}(t,x,v,K(t-),L(t),R(t,\cdot))+f(t,x,v)\big]dt,\quad t\in(\negthickspace(\tau_{k}\wedge T,\tau_{k+1}\wedge T)\negthickspace).\end{split}
\]
Since $k$ is a predictable process, it does not have a jump at the
inaccessible time $\tau_{k}$. Thus $dk$ does not contain singular
measure, in other word, any $t\in[0,T]$, 
\begin{equation}
\langle dk(t)x,x\rangle=\min_{v\in\mathbb{R}^{n}}\big[\mathbb{F}(t,x,v,K(t-),L(t),R(t,\cdot))+f(t,x,v)\big]dt.\label{eq:estima-2}
\end{equation}
In view of assertion (ii) in Proposition \ref{prop:quadr}, the right
hand side of \eqref{eq:estima-2} has a unique minimal point $\bar{u}(t)$,
hence the minmium value is nothing but $G(t,K(t-),L(t),R(t,\cdot))$
and $\mathscr{N}(t,K(t),R(t,\cdot))$ is invertible, which together
with \eqref{eq:3.16} implies that $(K,L,R)$ satisfies \eqref{eq:integ-Riccati}
a.s. \bigskip{}

\textit{Next we prove the BMO martingale property and \eqref{eq:exist-0}}.
Using \eqref{eq:exist-3} for $u(\cdot)=0$ and $X=X^{\tau,x;0}$,
we have 
\begin{equation}
\left\{ \begin{array}{ll}
dV(t,X(t)))= & -\langle dk(t)X(t),X(t)\rangle+\mathbb{F}(t,X(t-),0,K(t-),L(t),R(t,\cdot))dt\\
 & +{\displaystyle \sum_{i=1}^{d}\bigg[\langle L^{i}(t)X(t-),X(t-)\rangle+2\langle K(t-)X(t-),C^{i}(t)X(t-)\rangle\bigg]dW^{i}(s)}\\
 & +{\displaystyle \int_{\Lambda}\langle R(t,e)(X(t-)+E(t,e)X(t-)),X+E(t,e)X(t-)\rangle\tilde{\mu}(de,dt)}\\
 & +{\displaystyle \int_{\Lambda}\langle K(t-)E(t,e)X(t-),E(t,e)X(t-)\rangle\tilde{\mu}(de,dt)}\\
 & +2{\displaystyle \int_{\Lambda}\langle K(t-)X(t-),E(t,e)X(t-))\rangle\tilde{\mu}(de,dt)},\\
V(T,X(T))= & \langle MX(T),X(T)\rangle\text{.}
\end{array}\right.\label{eq:3.1}
\end{equation}
Applying Itô formula to $|V(t,X(t))|^{2}$, we have

\begin{align}
 & \int_{\tau}^{T}\int_{\Lambda}\bigg|\langle R(I+E)X,(I+E)X\rangle+\langle K(2I+E)X,EX\rangle\bigg|^{2}\mu(de,dt)\nonumber \\
 & \quad+\int_{\tau}^{T}\sum_{i=1}^{d}\bigg|\langle L^{i}X,X\rangle+2\langle KX,C^{i}X\rangle\bigg|^{2}dt\nonumber \\
= & |\langle MX(T),X(T)\rangle|^{2}-|\langle K(\tau)x,x\rangle|^{2}+2\int_{\tau}^{T}\langle KX,X\rangle\Big[\langle dkX,X\rangle-\mathbb{F}(t,X,0,K,L,R)dt\Big]\nonumber \\
 & \quad-2\int_{\tau}^{T}\int_{\Lambda}\langle KX,X\rangle\Big[\langle R(I+E)X,(I+E)X\rangle+\langle K(2I+E)X,EX\rangle\bigg]\tilde{\mu}(de,dt)\label{eq:exist-5}\\
 & \quad-2\int_{\tau}^{T}\langle KX,X\rangle{\displaystyle \sum_{i=1}^{d}\bigg[\langle L^{i}X,X\rangle+2\langle KX,C^{i}X\rangle\bigg]dW^{i}(t)},\nonumber 
\end{align}
where $X$ means $X^{\tau,x;0}(t-)$, $K$ means $K(t-)$.

In the following estimates the constant $C$ may change line by line.
Since $V(t,X(t))>0$ and the measure $\Pi(dt;\tau\wedge T,x,u)dxd\mathbb{P}$
(see \eqref{eq:4.10-1}) is nonnegative, we have a.e.
\begin{alignat*}{1}
 & \int_{\tau}^{T}\int_{\Lambda}\bigg|\langle R(I+E)X,(I+E)X\rangle+\langle K(2I+E)X,EX\rangle\bigg|^{2}\mu(de,dt)\\
 & \quad+\int_{\tau}^{T}\sum_{i=1}^{d}\bigg|\langle L^{i}X,X\rangle+2\langle KX,C^{i}X\rangle\bigg|^{2}dt\\
\leq & |M|^{2}|X(T)|^{4}+2\int_{\tau}^{T}\langle KX,X\rangle f(t,X,0)dt\\
 & \quad-2\int_{\tau}^{T}\int_{\Lambda}\langle KX,X\rangle\Big[\langle R(I+E)X,(I+E)X\rangle+\langle K(2I+E)X,EX\rangle\bigg]\tilde{\mu}(de,dt)\\
 & \quad-2\int_{\tau}^{T}\langle KX,X\rangle{\displaystyle \sum_{i=1}^{d}\bigg[\langle L^{i}X,X\rangle+2\langle KX,C^{i}X\rangle\bigg]dW^{i}(t)}.
\end{alignat*}

Thanks to inequality $\frac{1}{2}a^{2}-b^{2}\leq(a+b)^{2}$, and the
boundness of $K$ , we have 

\begin{align}
 & \int_{\tau}^{T}\int_{\Lambda}\bigg|\langle R(I+E)X,(I+E)X\rangle\bigg|^{2}\mu(de,dt)+\int_{\tau}^{T}\sum_{i=1}^{d}\bigg|\langle L^{i}X,X\rangle\bigg|^{2}dt\nonumber \\
\leq & 2\int_{\tau}^{T}\int_{\Lambda}\bigg|\langle K(2I+E)X,EX\rangle\bigg|^{2}\mu(de,dt)+2\int_{\tau}^{T}\sum_{i=1}^{d}\bigg|2\langle KX,C^{i}X\rangle\bigg|^{2}dt\nonumber \\
 & +2|M|^{2}|X|^{4}+4\int_{\tau}^{T}\langle KX,X\rangle f(t,X,0)dt\nonumber \\
 & -4\int_{\tau}^{T}\langle KX,X\rangle\sum_{i=1}^{d}\bigg[\langle L^{i}X,X\rangle+2\langle KX,C^{i}X\rangle\bigg]dW^{i}(t)\nonumber \\
 & -4\int_{\tau}^{T}\int_{\Lambda}\langle KX,X\rangle\Big[\langle R(I+E)X,(I+E)X\rangle+\langle K(2I+EX,EX\rangle\bigg]\tilde{\mu}(de,dt)\label{eq:exist-6-0}\\
\leq & C\sup_{\tau\leq t\leq T}|X|^{4}+\Big|4\int_{\tau}^{T}\langle KX,X\rangle\sum_{i=1}^{d}\langle L^{i}X,X\rangle d^{i}W(t)\Big|\nonumber \\
 & +8\Big|\int_{\tau}^{T}\langle KX,X\rangle\sum_{i=1}^{d}\langle KX,C^{i}X\rangle dW^{i}(t)\Big|\nonumber \\
 & +4\Big|\int_{\tau}^{T}\int_{\Lambda}\langle KX,X\rangle\langle R(I+E)X,(I+E)X\rangle\tilde{\mu}(de,dt)\Big|\nonumber \\
 & +4\Big|\int_{\tau}^{T}\int_{\Lambda}\langle KX,X\rangle\langle K(2I+E)X,E)X)\rangle\tilde{\mu}(de,dt)\Big|.\nonumber 
\end{align}
 This means that 
\begin{align}
 & \mathbb{E}^{\mathscr{F}_{\tau}}\bigg[\int_{\tau}^{T}\int_{\Lambda}\bigg|\langle R(I+E)X,(I+E)X\rangle\bigg|^{2}\mu(de,dt)\bigg]+\mathbb{E}^{\mathscr{F}_{\tau}}\bigg[\int_{\tau}^{T}\sum_{i=1}^{d}\bigg|\langle L^{i}X,X\rangle\bigg|^{2}dt\bigg]\nonumber \\
\leq & C\mathbb{E}^{\mathscr{F}_{\tau}}\bigg[\sup_{\tau\leq t\leq T}|X(t)|^{4}\bigg]+C_{p}\mathbb{E}^{\mathscr{F}_{\tau}}\bigg[\bigg|\int_{\tau}^{T}\langle KX,X\rangle{\displaystyle \sum_{i=1}^{d}\langle L^{i}X,X\rangle dW^{i}(t)\bigg|\bigg]}\label{eq:exist-6}\\
 & +C\mathbb{E}^{\mathscr{F}_{\tau}}\bigg[\bigg|\int_{\tau}^{T}\langle KX,X\rangle{\displaystyle \sum_{i=1}^{d}\langle KX,C^{i}X\rangle dW^{i}(t)\bigg|\bigg]}\nonumber \\
 & +C\mathbb{E}^{\mathscr{F}_{\tau}}\bigg[\bigg|\int_{\tau}^{\tau}\int_{\Lambda}\langle KX,X\rangle\langle R(I+E)X,(I+E)X\rangle\tilde{\mu}(de,dt)\bigg|\bigg]\nonumber \\
 & +C\mathbb{E}^{\mathscr{F}_{\tau}}\bigg[\bigg|\int_{\tau}^{T}\int_{\Lambda}\langle KX,X\rangle\langle K(I+E)X,EX\rangle\tilde{\mu}(de,dt)\bigg|\bigg].\nonumber 
\end{align}
Using BDG inequality, H\"older inequality, boundness of $K$ and
the estimation Lemma \ref{lem:3.1}, we have the following estimation
about the every terms in right hand side of \eqref{eq:exist-6-0},
\begin{equation}
\begin{split} & \mathbb{E}^{\mathscr{F}_{\tau}}\bigg[\bigg|\int_{\tau}^{T}\langle KX,X\rangle\sum_{i=1}^{d}\langle L^{i}X,X\rangle dW(t)\bigg|\bigg]\\
\leq & \mathbb{E}^{\mathscr{F}_{\tau}}\bigg[\bigg|\int_{\tau}^{T}\bigg|\langle KX,X\rangle{\displaystyle \sum_{i=1}^{d}\langle L^{i}X,X\rangle\bigg|^{2}dt\bigg|^{\frac{1}{2}}\bigg]}\\
\leq C & \mathbb{E}^{\mathscr{F}_{\tau}}\bigg[\sup_{\tau\leq t\leq T}|X|^{2}\bigg(\int_{\tau}^{T}{\displaystyle \sum_{i=1}^{d}\bigg|\langle L^{i}X,X\rangle\bigg|^{2}dt\bigg)^{\frac{1}{2}}\bigg]}\\
\leq\frac{C}{\varepsilon} & \mathbb{E}^{\mathscr{F}_{\tau}}\bigg[\sup_{\tau\leq T\leq T}|X|^{4}\bigg]+\varepsilon\mathbb{E}^{\mathscr{F}_{\tau}}\bigg[\bigg(\int_{\tau}^{T}\sum_{i=1}^{d}\Big|\langle L^{i}X,X\rangle\Big|^{2}dt\bigg)\bigg]\\
\leq & \frac{C}{\varepsilon}|x|^{4}+\varepsilon\mathbb{E}^{\mathscr{F}_{\tau}}\bigg[\int_{\tau}^{T}{\displaystyle \sum_{i=1}^{d}\Big|\langle L^{i}X,X\rangle\Big|^{2}dt\bigg]},
\end{split}
\label{eq:9.7}
\end{equation}

\begin{equation}
\begin{split} & \mathbb{E}^{\mathscr{F}_{\tau}}\bigg[\bigg|\int_{\tau}^{T}\langle KX,X\rangle{\displaystyle \sum_{i=1}^{d}\langle KX,C^{i}X\rangle dW^{i}(t)\bigg|\bigg]}\\
\leq & \mathbb{E}^{\mathscr{F}_{\tau}}\bigg[\bigg(\int_{\tau}^{T}\bigg|\langle KX,X\rangle{\displaystyle \sum_{i=1}^{d}\langle KX,C^{i}X\rangle\bigg|^{2}dt\bigg)^{\frac{1}{2}}\bigg]}\\
\leq & C_{p}E^{\mathscr{F}_{\tau}}\bigg[\sup_{\tau\leq t\leq T}|X|^{4}\bigg]\\
\leq & C_{p}|x|^{4},
\end{split}
\label{eq:9.6}
\end{equation}
\begin{equation}
\begin{split} & \mathbb{E}^{\mathscr{F}_{\tau}}\bigg[\bigg|\int_{\tau}^{T}\int_{\Lambda}\langle KX,X\rangle\langle K(2I+EX,EX)\rangle\tilde{\mu}(de,dt)\bigg|\bigg]\\
\leq & \mathbb{E}^{\mathscr{F}_{\tau}}\bigg[\Big(\int_{\tau}^{T}\int_{\Lambda}\bigg|\langle KX,X\rangle\langle K(2I+E)X,EX\rangle\bigg|^{2}\mu(de,dt)\Big)^{\frac{1}{2}}\bigg]\\
\leq & C\mathbb{E}^{\mathscr{F}_{\tau}}\bigg[\sup_{\tau\leq t\leq T}|X|^{4}\bigg]\\
\leq & C|x|^{4},
\end{split}
\label{eq:9.8}
\end{equation}

\begin{equation}
\begin{split} & \mathbb{E}^{\mathscr{F}_{\tau}}\bigg[\bigg|\int_{\tau}^{T}\int_{\Lambda}\langle KX,X\rangle\langle R(I+E)X,(I+E)X\rangle\tilde{\mu}(de,dt)\bigg|\bigg]\\
\leq & C\mathbb{E}^{\mathscr{F}_{\tau}}\bigg[\bigg\{\int_{\tau}^{T}\int_{\Lambda}\bigg|\langle KX,X\rangle\langle R(I+E)X,(I+E)X\rangle\bigg|^{2}\mu(de,dt)\bigg\}^{\frac{1}{2}}\bigg]\\
\leq & C\mathbb{E}^{\mathscr{F}_{\tau}}\bigg[\sup_{\tau\leq t\leq T}|X|^{2}\bigg\{\int_{\tau}^{T}\int_{\Lambda}\bigg|\langle R(I+E)X,(I+E)X\rangle\bigg|^{2}\mu(de,dt)\bigg\}^{\frac{1}{2}}\bigg]\\
\leq & \frac{C}{\varepsilon}\mathbb{E}^{\mathscr{F}_{\tau}}\bigg[\sup_{\tau\leq t\leq T}|X|^{4}\bigg]+\varepsilon\mathbb{E}^{\mathscr{F}_{\tau}}\bigg[\int_{\tau}^{T}\int_{\Lambda}\bigg|\langle R(I+E)X,(I+E)X\rangle\bigg|^{2}\mu(de,dt)\bigg]\\
\leq & \frac{C}{\varepsilon}|x|^{4}+\varepsilon\mathbb{E}^{\mathscr{F}_{\tau}}\bigg[\int_{\tau}^{T}\int_{\Lambda}\bigg|\langle R(I+E)X,(I+E)X\rangle\bigg|^{2}\mu(de,dt)\bigg].
\end{split}
\label{eq:9.9}
\end{equation}
Taking conditional expectation on both sides of \eqref{eq:exist-6-0},
putting \eqref{eq:9.6}-\eqref{eq:9.9} into it, and then letting
$\varepsilon=1/4$, we get

\[
\mathbb{E}^{\mathscr{F}_{\tau}}\bigg[\int_{\tau}^{T}\int_{\Lambda}\bigg|\langle R(I+E)X,(I+E)X\rangle\bigg|^{2}\mu(de,dt)\bigg]+\mathbb{E}^{\mathscr{F}_{\tau}}\bigg[\int_{\tau}^{T}\sum_{i=1}^{d}\bigg|\langle L^{i}X,X\rangle\bigg|^{2}dt\bigg]\leq C|x|^{4},
\]
 the constant $C$ is independent of $\tau$ and $x$. Then we have
\begin{equation}
\mathbb{E}^{\mathscr{F}_{\tau}}\bigg[\int_{\tau}^{T}\int_{E}\bigg|\Phi^{*}(I+E)^{*}R(I+E)\Phi\bigg|^{2}\mu(de,dt)\bigg]+\mathbb{E}^{\mathscr{F}_{\tau}}\bigg[\int_{\tau}^{T}\sum_{i=1}^{d}\bigg|\Phi L^{i}\Phi\bigg|^{2}dt\bigg]\leq C,\label{eq:exist-11}
\end{equation}
 where $\Phi$ is the solution of matrix equation \eqref{eq:sm-3-0}
on $[\![\tau\wedge T,T]\!]$ with initial data $\Phi(\tau\wedge T)=I$.

Recall that $\tau_{k}$ as the $k$-th jump time of the Poisson point
process. For any stopping time $\gamma\leq T$, denote by $\hat{\tau}_{k}:=\gamma\vee\tau_{k}$
the $n$-th jump time after the stopping time $\gamma$. Applying
\eqref{eq:exist-11} for $\tau=\hat{\tau}_{k}\wedge T$ and noting
that $\Phi$ is inversible on time $[\![\hat{\tau}_{k}\wedge T,\hat{\tau}_{k+1}\wedge T)\negmedspace)$
and $\mathbb{E}^{\mathscr{F}_{\hat{\tau}_{k}\wedge T}}\bigg[\sup_{t\in[\![\hat{\tau}_{k}\wedge T,\hat{\tau}_{k+1}\wedge T)\negthickspace)}\Phi^{-4}(t)\bigg]$
is bounded by a constant only depending on the bound of the coefficients
and $T$ (see \eqref{eq:sm-3} for details), we see that for any $k\geq1$,
\begin{align}
 & \mathbb{E}^{\mathscr{F}_{\hat{\tau}_{k}\wedge T}}\bigg[\int_{\hat{\tau}_{k}\wedge T}^{\hat{\tau}_{k+1}\wedge T}\sum_{i}|L^{i}|^{2}dt\bigg]\nonumber \\
\leq & \mathbb{E}^{\mathscr{F}_{\hat{\tau}_{k}\wedge T}}\bigg[\int_{\hat{\tau}_{k}\wedge T}^{\hat{\tau}_{k+1}\wedge T}\sum_{i}|(\Phi^{*})^{-1}\Phi^{*}L^{i}\Phi\Phi^{-1}|^{2}dt\bigg]\nonumber \\
\leq & \mathbb{E}^{\mathscr{F}_{\hat{\tau}_{k}\wedge T}}\bigg[\sup_{t\in[\![\hat{\tau}_{k}\wedge T,\hat{\tau}_{k+1}\wedge T)\negthickspace)}|\Phi^{-1}(t)|^{2}\int_{\hat{\tau}_{k}\wedge T}^{\hat{\tau}_{k+1}\wedge T}\sum_{i}|\Phi L^{i}\Phi|^{2}dt\bigg]\label{eq:exist-11+1}\\
\leq & \bigg\{\mathbb{E}^{\mathscr{F}_{\hat{\tau}_{k}\wedge T}}\sup_{t\in[\![\hat{\tau}_{k}\wedge T,\hat{\tau}_{k+1}\wedge T)\negthickspace)}|\Phi^{-1}(t)|^{4}\bigg\}^{\frac{1}{2}}\bigg\{\mathbb{E}^{\mathscr{F}_{\hat{\tau}_{k}\wedge T}}\bigg[\int_{\hat{\tau}_{k}\wedge T}^{\hat{\tau}_{k+1}\wedge T}\sum_{i}|\Phi^{*}L^{i}\Phi|^{2}dt\bigg]\bigg\}^{\frac{1}{2}}\nonumber \\
\leq & C.\nonumber 
\end{align}
Similarly, we have 
\begin{equation}
\mathbb{E}^{\mathscr{F}_{\gamma}}\bigg[\int_{\gamma\wedge T}^{\hat{\tau}_{1}\wedge T}\sum_{i}|L^{i}|^{2}dt\bigg]\le C.\label{eq:exist-11+2}
\end{equation}
Then, using estimates \eqref{eq:exist-11+1} and \eqref{eq:exist-11+2},
we have
\[
\begin{split} & \mathbb{E}^{\mathscr{F}_{\gamma}}\bigg[\int_{\gamma}^{T}\sum_{i}|L^{i}|^{2}dt\bigg]\\
= & \mathbb{E}^{\mathscr{F}_{\gamma}}\bigg[(\int_{\gamma}^{\hat{\tau}_{1}\wedge T}+\sum_{n=1}^{\infty}\int_{\hat{\tau}_{n}\wedge T}^{\hat{\tau}_{n+1}\wedge T})\sum_{i}|L^{i}|^{2}dt\bigg]\\
= & \mathbb{E}^{\mathscr{F}_{\gamma}}\Big[\int_{\gamma}^{\hat{\tau}_{1}\wedge T}\sum_{i}|L^{i}|^{2}dt\Big]+\sum_{n=1}^{\infty}\mathbb{E}^{\mathscr{F}_{\gamma}}\Big[\int_{\hat{\tau}_{n}\wedge T}^{\hat{\tau}_{n+1}\wedge T}\sum_{i}|L^{i}|^{2}dt\Big]\\
= & \mathbb{E}^{\mathscr{F}_{\gamma}}\bigg[\chi_{\{\gamma<T\}}\mathbb{E}^{\mathscr{F}_{\gamma}}\Big[\int_{\gamma}^{\hat{\tau}_{1}\wedge T}\sum_{i}|L^{i}|^{2}dt\Big]+\sum_{n=1}^{\infty}\chi_{\{\hat{\tau}_{n}<T\}}\mathbb{E}^{\mathscr{F}_{\hat{\tau}_{n}\wedge T}}\Big[\int_{\hat{\tau}_{n}\wedge T}^{\hat{\tau}_{n+1}\wedge T}\sum_{i}|L^{i}|^{2}dt\Big]\bigg]\\
\leq & \mathbb{E}^{\mathscr{F}_{\gamma}}\bigg[\Big(\chi_{\{\gamma<T\}}+\sum_{n=1}^{\infty}\chi_{\{\hat{\tau}_{n}<T\}}\Big)C\bigg]\\
= & C\mathbb{E}^{\mathscr{F}_{\gamma}}\Big[\mu\big([\![\gamma,T]\!]\times\Lambda\big)\Big].
\end{split}
\]
 In view of the independent increment property of the Poisson point
process $\{p_{t}\}_{t\geq0}$, $\mu\big([\![\gamma,T]\!]\times\Lambda\big)$
is independent of $\mathscr{F}_{\gamma}$. So we have $\mathbb{E}^{\mathscr{F}_{\gamma}}\Big[\mu\big([\![\gamma,T]\!]\times\Lambda)\big)\Big]=\mathbb{E}\Big[\mu\big([\![\gamma,T]\!]\times\Lambda\big)\Big]\leq\mathbb{E}\Big[\mu\big([0,T]\times\Lambda\big)\Big]=T\nu(\Lambda)$.
Hence we obtain that for any stopping time $\gamma$ valued in $[0,T]$,
\begin{equation}
\mathbb{E}^{\mathscr{F}_{\gamma}}\Big[\Big|\sum_{i}\int_{\gamma}^{T}L^{i}dW^{i}(t)\Big|^{2}\Big]\leq\mathbb{E}^{\mathscr{F}_{\gamma}}\bigg[\int_{\gamma}^{T}\sum_{i}|L^{i}|^{2}dt\bigg]\leq C,\label{eq:exist-12-1}
\end{equation}
which means $\int_{0}^{\cdot}L^{i}(s)dW_{s}^{i}$ is a BMO martingale,
$i=1,\ldots,d$.

For $\eta_{t}:=\int_{0}^{t}\int_{\Lambda}R(e,t)\tilde{\mu}(de,dt)$,
we see that it is a purely continuous martingale whose jumps coincide
with those of $K$. Since $K$ is uniformly bounded by some constant
$\lambda$, jumps of $\eta$ is also uniformly bounded by $2\lambda$.
Hence we have 
\begin{align*}
[\eta]_{T}-[\eta]_{\gamma-}= & \sum_{\gamma\leq s\le T}|\Delta\eta_{s}|^{2}=\sum_{\gamma\leq\tau_{i}\leq T}|\Delta\eta_{s}|^{2}\\
\leq & 4\lambda^{2}\mu([\![\gamma,T]\!]\times\Lambda).
\end{align*}
 Thus 
\begin{equation}
\mathbb{E}^{\mathscr{F}_{\gamma}}\bigg[[\eta]_{T}-[\eta]_{\gamma-}\bigg]\le4\lambda^{2}\mathbb{E}^{\mathscr{F}_{\gamma}}\bigg[\mu([\![\gamma,T]\!]\times\Lambda)\bigg]\le CT\nu(\Lambda)<\infty,\label{eq:exist-14}
\end{equation}
 which means that $J$ is also a BMO martingale. 

Let $\gamma=0$ in \eqref{eq:exist-12-1} and \eqref{eq:exist-14},
we have

\begin{equation}
\mathbb{E}\bigg[\int_{0}^{T}\sum_{i}|L^{i}|^{2}dt\bigg]\le C,\label{eq:exist-13}
\end{equation}
and
\begin{equation}
\mathbb{E}\bigg[\int_{0}^{T}\int_{\Lambda}R^{2}\nu(de)dt\bigg]=\mathbb{E}[\eta]_{T}\leq C,\label{eq:exist-15}
\end{equation}
 we have estimate \eqref{eq:exist-0}. 

\bigskip{}

\textit{Last we show the nonnegativity of $\int_{\Lambda}F^{*}(t,e)(K(t-)+R(t,e))F(t,e)\nu(de)$}.
First note that the pure jump process $\zeta_{t}:=\int_{0}^{t}\int_{\Lambda}F^{*}(s,e)(K(t-)+R(s,e))F(s,e)\mu(de,ds)$
only changes its value at the jumping time of Poisson process and
$\Delta\zeta_{t}=\int_{\Lambda}F^{*}(s,e)(K(s-)+R(s,e))F(s,e)\mu(de,\{t\})$.
Since at the jumping moment $R(s,p_{s})$ is equivalent to $K(s)-K(s-)$,
it is easy to know that $K(s-)+R(s,e)=K(s)$ is nonnegative definite
(here $e$ is the jumpping amplitude at the moment), therefore for
any $y\in{\mathcal{M}}_{{\cal \mathscr{F}}}^{\infty}(0,T;\mathbb{R}^{m})$,
we have 
\[
\int_{0}^{T}\int_{\Lambda}y^{*}(s)F^{*}(s,e)(K(s-)+R(s,e))F(s,e)y(s)\mu(de,ds)\geq0,\quad P{\rm -a.e.}
\]
In view of the martingale property, 

\[
\mathbb{E}\big[\int_{0}^{T}\int_{\Lambda}y^{*}(s)F^{*}(s,e)(K(s-)+R(s,e))F(s,e)y(s)\tilde{\mu}(ds,de)\big]=0.
\]
Hence 
\[
\begin{split} & \mathbb{E}\big[\int_{0}^{T}\int_{\Lambda}y^{*}(s)F^{*}(s,e)(K(s-)+R(s,e))F(s,e)y(s)\nu(de)ds\big]\\
= & \mathbb{E}\Big[\int_{0}^{T}\int_{\Lambda}y^{*}(s)F^{*}(s,e)(K(s-)+R(s,e))F(s,e)y(s)[\mu(de,ds)-\tilde{\mu}(de,ds)]\Big]\ge0.
\end{split}
\]
By the arbitrariness of $y$, we have $\int_{\Lambda}F^{*}(t,e)(K(t-)+R(t,e))F(t,e)\nu(de)$
is nonnegative for almost all $t$, $\mathbb{P}$-a.s. $\omega$.
Thus, the proof is complete.

\end{proof}

\begin{rmk}

If we have the condition \eqref{eq:addi-cond} in hand, \eqref{eq:exist-15}
could be obtained from \eqref{eq:exist-11} directly like the way
of \eqref{eq:exist-11+1}-\eqref{eq:exist-12-1}. In our case, observing
the structure of BSREJ and utilizing the relationship between the
jump of $K$ and $R$, we can prove \eqref{eq:exist-15} by the estimate
of $K$, and this way seemed to be easier.\end{rmk}

\section{\label{sec:Verification}Verification theorem}

In section \ref{sec:Exist}, we exploit Problem \ref{pro:4.1} and
the dynamic programming principle to show the existence of solution
for BSREJ \eqref{eq:Riccati}. In this section we will deal with the
problem from an inverse aspect -- if the BSREJ \eqref{eq:Riccati}
has a solution, how to describe the corresponding optimal control
problem? The following Theorem \ref{thm:verific} tells us that the
existence of solution for BSREJ \eqref{eq:Riccati} means the existence
of the optimal control for problem \eqref{eq:b8}. Besides, the optimal
control could be depicted as a linear feedback by the solution of
BSREJ \eqref{eq:Riccati}.

\begin{thm} \label{thm:verific}Let Assumptions \ref{ass:4.1} be
satisfied. And assume BSREJ \eqref{eq:Riccati} has a solution $(K,L,R)$
in the meaning of Definition \ref{def:solution}. Then the linear
SDE 
\begin{equation}
\left\{ \begin{array}{l}
d\bar{X}^{t,x}(s)=[A(s)-B(s)\mathscr{N}^{-1}(s,K(s-),R(s,\cdot))\mathscr{M}^{*}(s,K(s-),L(s),R(s,\cdot))]\bar{X}^{t,x}(s-)ds\\
\quad+\sum_{i=1}^{d}[C^{i}(s)-D^{i}(s)\mathscr{N}^{-1}(s,K(s-),R(s,\cdot))\mathscr{M}^{*}(s,K(s-),L(s),R(s,\cdot))]\bar{X}^{t,x}(s-)dW^{i}(s)\\
\quad+\int_{\Lambda}[E(s,e)-F(s,e)\mathscr{N}^{-1}(s,K(s-),R(s,\cdot))\mathscr{M}^{*}(s,K(s-),L(s),R(s,\cdot))]\bar{X}^{t,x}(s-)\tilde{\mu}(ds,de),\\
\bar{X}(t)=x,\qquad s\in[t,T]
\end{array}\right.\label{SDE: optimal feedback}
\end{equation}
has a unique solution $\bar{X}^{t,x}(\cdot)$ such that 
\begin{equation}
\mathbb{E}^{\mathscr{F}_{t}}\bigg[\sup_{s\in[t,T]}|\bar{X}^{t,x}(s)|^{2}\bigg]<C_{x},\label{integrable}
\end{equation}
where the constant $C_{x}$ is independent of initial time $t$.\\
(ii) The given process

\begin{equation}
\bar{u}^{t,x}(s):=-\mathscr{N}^{-1}(s,K(s-),R(s,\cdot))\mathscr{M}^{*}(s,K(s-),L(s),R(s,\cdot))\bar{X}(s-),\quad s\in[t,T]\label{eq:feedBackCont}
\end{equation}
belongs to ${\cal M}_{\mathscr{F}}^{2}(t,T;\mathbb{R}^{m})$, and
is the optimal control for the problem \eqref{eq:3.3} for the initial
data $(\tau,\xi)=(t,x)$.\\
 (iii) The value field $V$ is given by

\begin{equation}
V(t,x)=\langle K(t)x,x\rangle,(t,x)\in[0,T]\times\mathbb{R}^{n}.\label{eq:V-K}
\end{equation}

\end{thm} 

\begin{proof} Since the coefficients of the optimal SDE (\ref{SDE: optimal feedback})
are square integrable w.r.t. $t$ a.s., it admits a unique strong
solution $\bar{X}(\cdot)$. For a sufficiently large integer $j$,
define the stopping time $\gamma_{j}$ as follows: 
\[
\gamma_{j}^{t,x}:=T\wedge\inf\{s\ge t||\bar{X}^{t,x}(s)|\ge j\}
\]
with the convention that $\inf\emptyset=\infty$. It is obvious that
$\gamma_{j}^{t,x}\uparrow T$ almost surely as $j\uparrow\infty$.
Then by It\^o formula we have 
\begin{equation}
\langle K(t)x,x\rangle=\mathbb{E}^{\mathscr{F}_{t}}\bigg[\langle K(\gamma_{j}^{t,x})\bar{X}^{t,x}(\gamma_{j}^{t,x}),\bar{X}^{t,x}(\gamma_{j}^{t,x})\rangle+\int_{t}^{\gamma_{j}^{t,x}}f(s,\bar{X}^{t,x}(s),\bar{u}^{t,x}(s))dt\bigg].\label{eq:verif-3}
\end{equation}
Noting that $K$ is positive and bounded by $\lambda$, and $N>\delta I$
for some constant $\delta$ (see Assumption \ref{ass:4.1}), \eqref{eq:verif-3}
implies 
\[
\mathbb{E}^{\mathscr{F}_{t}}\bigg[\int_{t}^{\gamma_{j}^{t,x}}\big(\bar{u}^{t,x}\big)^{2}(s)ds\bigg]\le\frac{1}{\delta}\mathbb{E}^{\mathscr{F}_{t}}\bigg[\int_{t}^{\gamma_{j}^{t,x}}f(s,\bar{X}^{t,x}(s),\bar{u}^{t,x}(s))ds\bigg]\le\frac{1}{\delta}\langle K(t)x,x\rangle\leq\frac{\lambda}{\delta}|x|^{2}.
\]
Using Fatou's lemma, we have $\bar{u}^{t,x}(\cdot)\in{\cal M}_{\mathscr{F}}^{2}(0,T;\mathbb{R}^{m})$.
Then we have the estimation (\ref{integrable}) from Lemma \ref{lem:3.1}.
Thus, Assertion (i) and the first part of the assertion (ii) have
been proved.\bigskip{}

Now we prove the optimality of $\bar{u}^{t,x}(\cdot)$ and the assertion
(iii). By \eqref{integrable}, we know for any stopping time $\tau$
valued in $[t,T]$, 
\[
\mathbb{E}^{\mathscr{F}_{t}}\Big[\big|\bar{X}^{t,x}(\tau)\big|^{2}\Big]\leq\mathbb{E}^{\mathscr{F}_{t}}\bigg[\sup_{s\in[t,T]}|\bar{X}^{t,x}(s)|^{2}\bigg]<C_{x},
\]
 hence $|\bar{X}^{t,x}|^{2}$ is uniformly integrable. Besides \eqref{integrable}
together with Chebyshev inequality shows that for any positive integer
$j$,
\[
\mathbb{P}\bigg(\sup_{s\in[t,T]}|\bar{X}^{t,x}(s)|\geq j\bigg)\leq\frac{\mathbb{E}^{\mathscr{F}_{t}}\bigg[\sup_{s\in[t,T]}|\bar{X}^{t,x}(s)|^{2}\bigg]}{j^{2}}\to0,\qquad{\rm as}\,j\to\infty.
\]
 It follows that $\,\mathbb{P}\{\gamma_{j}^{t,x}=T\}\nearrow1$. Combining
the dominate convergence theorem and the boundness of $K$, we have
the first term in right hand of \eqref{eq:verif-3} $\mathbb{E}^{\mathscr{F}_{t}}\bigg[\langle K(\gamma^{t,x})\bar{X}(\gamma^{t,x}),\bar{X}(\gamma^{t,x})\rangle\bigg]\to\mathbb{E}^{\mathscr{F}_{t}}\bigg[\langle K(T)\bar{X}^{t,x}(T),\bar{X}^{t,x}(T)\rangle\bigg]$
as $j\to\infty$. The $L^{2}$-boundness of $\bar{X}^{t,x}(\cdot)$
and $\bar{u}^{t,x}(\cdot)$ yields the second term in right hand of
\eqref{eq:verif-3} $\mathbb{E}^{\mathscr{F}_{t}}\bigg[\int_{t}^{\gamma_{j}^{t,x}}f(s,\bar{X}^{t,x}(s),\bar{u}^{t,x}(s))ds\bigg]\to\mathbb{E}^{\mathscr{F}_{t}}\bigg[\int_{t}^{T}f(s,\bar{X}^{t,x}(s),\bar{u}^{t,x}(s))ds\bigg]$
as $j\to\infty$. Hence \eqref{eq:verif-3} yields
\begin{align}
 & \langle K(t)x,x\rangle\nonumber \\
= & \lim_{j\to\infty}\mathbb{E}^{\mathscr{F}_{t}}\bigg[\langle K(\gamma_{j}^{t,x})\bar{X}^{t,x}(\gamma_{j}^{t,x}),\bar{X}^{t,x}(\gamma_{j}^{t,x})\rangle+\int_{t}^{\gamma_{j}^{t,x}}f(s,\bar{X}^{t,x}(s),\bar{u}^{t,x}(s))ds\bigg]\label{eq:verify-4}\\
= & \mathbb{E}^{\mathscr{F}_{t}}\bigg[\langle K(T)\bar{X}^{t,x}(T),\bar{X}^{t,x}(T)\rangle+\int_{t}^{T}f(s,\bar{X}^{t,x}(s),\bar{u}^{t,x}(s))ds\bigg]=J(\bar{u}^{t,x}(\cdot);0,x).\nonumber 
\end{align}

To obtain the optimality of $\bar{u}^{t,x}(\cdot)$, it remain to
show 
\[
J(u(\cdot);t,x)\geq\langle K(t)x,x\rangle,\quad\forall u(\cdot)\in{\cal M}_{\mathscr{F}}^{2}(t,T;\mathbb{R}^{m}).
\]
 To do this, for any $u(\cdot)\in{\cal M}_{\mathscr{F}}^{2}(t,T;\mathbb{R}^{m})$,
define the stopping times 
\[
\gamma_{j}^{t,x;u(\cdot)}=T\wedge\inf\{s\ge t||X^{t,x;u(\cdot)}(s)|\ge j\},\quad j\in\mathbb{Z}_{+}.
\]
 Same as $\gamma_{j}^{t,x}$, $\gamma_{j}^{t,x;u(\cdot)}\nearrow T$
and $\mathbb{P}\{\gamma_{j}^{t,x;u(\cdot)}=T\}\nearrow1$ as $j\to\infty$.
Define 
\[
\tilde{u}(s):=-\mathscr{N}^{-1}(s,K(s-),R(s,\cdot))\mathscr{M}(s,K(s-),L(s),R(s,\cdot))X^{0,x,u(\cdot)}(s-),\quad s\in[t,T].
\]
 Obviously, $\mathbb{E}\Big[\int_{t}^{\gamma_{j}^{t,x;u(\cdot)}}\big|\tilde{u}(t)\big|^{2}dt\Big]<\infty$.
Then applying It\^o formula to $\langle K(t)X^{t,x;u(\cdot)}(t),X^{t,x;u(\cdot)}(t)\rangle$
and by  straightforward computing, we get that
\begin{align}
 & \mathbb{E}^{\mathscr{F}_{t}}\bigg[\langle K(\gamma_{j}^{tx;u(\cdot)})X^{t,x;u(\cdot)}(\gamma_{j}^{t,x;u(\cdot)}),X^{t,x;u(\cdot)}(\gamma_{j}^{t,x;u(\cdot)})\rangle+\int_{t}^{\gamma_{j}^{t,x;u(\cdot)}}f(s,X^{t,x;u(\cdot)}(s),u(s))ds\bigg]\label{eq:verify-5}\\
= & \langle K(t)x,x\rangle+\mathbb{E}^{\mathscr{F}_{t}}\Big[\int_{t}^{\gamma_{j}^{t,x;u(\cdot)}}\big\langle\mathscr{N}^{-1}(s,K(s),R(s,\cdot))\big(u(s)-\tilde{u}(s)\big),u(s)-\tilde{u}(s)\big\rangle\Big]\nonumber \\
\geq & \langle K(t)x,x\rangle.\nonumber 
\end{align}
Since $u(\cdot)\in{\cal M}_{\mathscr{F}}^{2}(t,T;\mathbb{R}^{m})$,
according to the estimate \eqref{integrable}, similar to the limitation
in \eqref{eq:verify-4}, we take limit in \eqref{eq:verify-5}
\begin{align*}
 & J(u(\cdot);t,x)\\
= & \mathbb{E}^{\mathscr{F}_{t}}\bigg[\langle K(T)X^{t,x;u}(T),X^{t,x;u}(T)\rangle+\int_{t}^{T}f(s,X^{t,x;u}(s),u(s))ds\bigg]\\
\geq & \langle K(t)x,x\rangle.
\end{align*}

\end{proof}

According to the above verification theorem, we immediately have the
following uniqueness of the solution for BSREJ \eqref{eq:Riccati}.

\begin{thm}

Let Assumptions \ref{ass:4.1} be satisfied. Let $(\tilde{K},\tilde{L},\tilde{R})$
be another solution of BSREJ \eqref{eq:Riccati} in the meaning of
Definition \ref{def:solution}. Then $(\tilde{K},\tilde{L},\tilde{R})=(K,L,R)$.

\end{thm}

\begin{proof}

In view of \eqref{eq:V-K}, the uniqueness of value function $V$
leads to that of first unknown variable $K$ of solution for BSREJ
\eqref{eq:Riccati}, hence $\tilde{K}=K$. By the expression of BSREJ
\eqref{eq:Riccati}, the integration w.r.t. $\tilde{\mu}$ is just
pure jump martingale, hence  
\begin{align*}
\sum_{s\leq t}\Delta K_{s} & =\int_{0}^{t}\int_{\Lambda}R(t,e)\mu(de,ds),\\
\sum_{s\leq t}\Delta\tilde{K}_{s} & =\int_{0}^{t}\int_{\Lambda}\tilde{R}(t,e)\mu(de,ds).
\end{align*}
Comparing the above two equality, taking the quadratic variation (the
bracket) , and then taking expectation on both sides, we have 
\[
0=\mathbb{E}\Big[\sum_{s\leq t}\Delta\big(K_{s}-\tilde{K}_{s}\big),\sum_{s\leq t}\Delta\big(K_{s}-\tilde{K}_{s}\big)\Big]=\mathbb{E}\Big[\int_{0}^{t}\int_{\Lambda}\big(R-\tilde{R}\big)^{2}\nu(de)ds\Big].
\]
 This means $\tilde{R}=R$. 

With the uniqueness of the first and third unknown variables $(K,R)$
in hand, the uniqueness of the optimal control and its feedback form
\eqref{eq:feedBackCont} yields the uniqueness of the second unknown
variable $L$.\end{proof}

\bibliographystyle{plain}
\bibliography{LQJ}

\end{document}